\documentclass[12pt,sort&compress,1p]{elsarticle}

\usepackage[utf8]{inputenc}
\usepackage[T1]{fontenc}
\usepackage[english]{babel}
\usepackage{amssymb,amsmath,amsthm,amsfonts}
\usepackage{nicefrac}
\usepackage{subfig}

\theoremstyle{plain}
\newtheorem{theorem}{Theorem}
\newtheorem{proposition}{Proposition}
\theoremstyle{definition}

\newcommand{\func}[1]{\operatorname{#1}}

\usepackage{lineno}
\usepackage{hyperref}
\hypersetup{breaklinks=true}
\usepackage{matlab-prettifier}
\journal{Journal of Computational Physics}

\begin{document}

\begin{frontmatter}

\title{Rational Krylov methods for functions of matrices with applications
    to fractional partial differential equations\tnoteref{funding}}
\tnotetext[funding]{The authors are members of the INdAM Research group GNCS, and this work was supported in part by the
    INDAM-GNCS 2018 projects ``Tecniche innovative per problemi di
    algebra lineare'' and ``Risoluzione numerica di equazioni di
    evoluzione integrali e differenziali con memoria'', by the University of Pisa (Grant PRA\_2017\_05), and FRA-University of Trieste.
    D. Bertaccini gratefully acknowledges the MIUR Excellence Department Project
    awarded to the Department of Mathematics, University of Rome Tor
    Vergata, CUP E83C18000100006 and the Tor Vergata University \lq\lq
    MISSION: SUSTAINABILITY\rq\rq  project \lq\lq NUMnoSIDS\rq\rq, CUP
    E86C18000530005.}

\author[mymainaddress1]{L. Aceto}
\address[mymainaddress1]{
    Universit\`{a} di Pisa,
    Dipartimento di Matematica,
    Via F. Buonarroti, 1/C, Pisa, Italy}
\ead[url]{lidia.aceto@unipi.it}

\author[mymainaddress2,mymainaddress2bis]{D. Bertaccini\corref{mycorrespondingauthor}}
\address[mymainaddress2]{Universit\`{a} di Roma ``Tor Vergata'',
    Dipartimento di Matematica, viale della Ricerca Scientifica 1, Roma,
    Italy. Email: bertaccini@mat.uniroma2.it}
\address[mymainaddress2bis]{Istituto per le Applicazioni del Calcolo (IAC)
    \lq\lq M. Picone\rq\rq, National Research Council (CNR), Roma,
    Italy.} \ead[url]{bertaccini@mat.uniroma2.it}

\author[mymainaddress3]{F. Durastante}
\address[mymainaddress3]{Universit\`{a} di Pisa,
    Dipartimento di Informatica, largo Bruno Pontecorvo 3, Pisa, Italy.} \ead[url]{fabio.durastante@di.unipi.it}

\author[mymainaddress4]{P. Novati}
\address[mymainaddress4]{Universit\`{a} di Trieste,
    Department of Mathematics and Geosciences,
    via Valerio 12/1, 34127 Trieste, Italy}
\cortext[mycorrespondingauthor]{Corresponding author}

\ead[url]{novati@units.it}

\begin{abstract}
In this paper we propose a new choice of poles to define reliable
rational Krylov methods. These methods are used for approximating
function of positive definite matrices. In particular, the
fractional power and the fractional resolvent are considered because
of their importance in the numerical solution of fractional partial
differential equations. The numerical experiments on some fractional
partial differential equation models confirm that the proposed
approach is promising.
\end{abstract}

\begin{keyword}
Fractional Laplacian \sep Matrix functions \sep Krylov methods \sep Gauss-Jacobi rule
\MSC[2010] 65f60 \sep 35r11 \sep 65d32
\end{keyword}

\end{frontmatter}

%\linenumbers

\section{Introduction}
\label{sec:introduction}
This paper deals with some
computational issues concerning the use of Krylov-type methods for
computing certain functions of matrices occurring in the solution of
fractional partial differential equations.
In particular, we focus
on the following model problems:
\begin{equation}\label{eq:steady_state}
\left\lbrace \begin{array}{ll}
(-\Delta)^{\frac{\alpha}{2}} u = s(\mathbf{x},u), &  \mathbf{x} \in \Omega,\\
u(\mathbf{x}) = 0, & \mathbf{x} \in \partial \Omega,
\end{array}\right.
\end{equation}
and
\begin{equation}\label{eq:time_dependent}
\left\lbrace \begin{array}{ll}
\frac{\partial u}{\partial t} = - \mu(-\Delta)^{\frac{\alpha}{2}} u + s(\mathbf{x},u,t), & (\mathbf{x},t) \in \Omega \times (t_0,T],\\
u(\mathbf{x},t_0) = u_0(\mathbf{x}), & \mathbf{x} \in \Omega, \\
u(\mathbf{x},t) = 0 \quad \mbox{(or} \quad \partial_n u(\mathbf{x},t) = 0), & (\mathbf{x},t) \in \partial \Omega \times (t_0,T],
\end{array}\right.
\end{equation}
where $\alpha \in (1,2]$, $\mu>0,$ and  $-(-\Delta )^{\alpha /2}$ is the
fractional Laplacian operator with homogeneous Dirichlet (or Neumann)
boundary conditions. As usual, this operator can be defined using the
spectral decomposition of the Laplacian, that is,
\begin{equation*}
(-\Delta )^{\alpha /2}u=\sum_{p=1}^{\infty }\mu _{p}^{\alpha /2}c_{p}\varphi
_{p},\quad u=\sum_{p=1}^{\infty }c_{p}\varphi _{p},
\end{equation*}
where $\{\varphi _{p}\}_{p=1}^{\infty }$ are the eigenfunctions of $(-\Delta
)$ and $\{\mu _{p}\}_{p=1}^{\infty }$ are the corresponding positive real
eigenvalues. In this view, any positive definite matrix representing an appropriate discretization $%
A_{n}\in \mathbb{R}^{n\times n}$ of the Laplacian operator (that is, $%
\mathbf{x}^{T}A_{n}\mathbf{x}>0$ for any nonzero $\mathbf{x}\in \mathbb{C}^{n}$) can be used to approximate $%
(-\Delta )^{\alpha /2}$ by means of $A_{n}^{\alpha /2}.$ This approach
is known as \emph{Matrix Transfer
Technique} (MTT)~\cite{MR2970400,MR3739307,ilic2005numerical,ilic2006numerical,yang2011novel,yang2014finite}.
Here we use finite differences, finite volumes and finite elements methods
to generate the underlying approximation~$A_{n}$.

For what concerns the steady state problem \eqref{eq:steady_state},
the solution can be
approximated by any method able to compute the action of $A_{n}^{-\alpha /2}.$
On the other hand, the solution of \eqref{eq:time_dependent} requires also the
discretization in time. Since $A_{n}$ inherits the stiffness of the
Laplacian operator, the time integration of
\eqref{eq:time_dependent} needs to be implicit. Therefore, at each
time step, we need to compute the action of matrix functions of the type $%
\left( I+\nu A_{n}^{\alpha /2}\right) ^{-1},$
where $I$ denotes the  identity matrix of size $n$ and $\nu >0$ is a
parameter that depends on the stepsize and the integrator. Assuming
that the stiffness is only due to the Laplacian operator, here we
consider IMEX-type methods, which allow an explicit treatment of the
forcing term in \eqref{eq:time_dependent}.

In this work, we analyze some Krylov-type methods approximating the
action of the product of the above matrix functions by a given
vector. We compare the behavior of the polynomial, shift-and-invert,
extended, and rational Krylov methods in terms of accuracy and
computational cost on various test problems and different
discretizations. Our aim is to show that our rational Krylov method
can outperform the other approaches provided that the poles of the
underlying rational forms are suitably defined. To this purpose, we
define the poles using the Gauss-Jacobi approach for the computation
of $A_{n}^{-\alpha /2}$ as
in~\cite{2018arXiv180710086A,aceto2017efficient,aceto2017rational},
and then extend this idea for the computation of $\left( I+\nu
A_{n}^{\alpha /2}\right) ^{-1} \mathbf{v}$ giving also some details
for a reliable implementation. We prove that our new poles are real
and simple, so that they can be computed with a root-finder less
prone to error amplification. We remark that the computation of
$\left( I+\nu A_{n}^{\alpha /2}\right) ^{-1}\mathbf{v}$ with Krylov
methods has been considered in~\cite{MR3854059},
but by using only the shift-and-invert approach. \\

The paper is organized as follows. In
Section~\ref{sec:Rational_Krylov_Methods} we recall the basic
features of Krylov methods of polynomial and rational types tailored
for approximating function of matrices, while in
Section~\ref{sec:Optimal_Poles_Selection} we give a theoretical and
computational analysis of the poles of the proposed
rational methods. In
Section~\ref{sec:Numerical_Solution_of_the_Model_Problems} we
discuss some issues of the MTT based on finite differences, finite
elements, and finite volumes approximation applied to the Laplacian
operator. Finally, Section~\ref{sec:Numerical_Experiments} is
devoted to some numerical experiments in which we compare the
performances of our approach with the other Krylov methods presented
in Section~\ref{sec:Rational_Krylov_Methods}.

\section{Rational Krylov Methods}

\label{sec:Rational_Krylov_Methods}
We recall that for a given matrix $A_{n}\in \mathbb{R}^{n\times n}$ and a function $f$ that is analytic on and inside a closed contour $\Gamma $
that encloses the spectrum of $A_{n}$, the matrix function $f(A_{n})$ is defined
as
\begin{equation*}
f(A_{n})=\frac{1}{2\pi i}\int_{\Gamma }f(z)(zI -A_{n})^{-1}\,d\,z.
\end{equation*}

Let $V_{k}$ be an orthogonal matrix whose columns $\mathbf{v}_{1},\dots, \mathbf{v}_{k}$ span an arbitrary Krylov subspace $\mathcal{W}_{k}(A_{n},%
\mathbf{v})$ of dimension $k$. We obtain an approximation of $f(A_{n})%
\mathbf{v}$ by
\begin{equation}
\begin{split}
f(A_{n})\mathbf{v}=& \frac{1}{2\pi i}\int_{\Gamma }f(z)(zI -A_{n})^{-1}%
\mathbf{v}\,d\,z \\
\approx & \frac{1}{2\pi i}\int_{\Gamma
}f(z)V_{k}(zI -V_{k}^{T}A_{n}V_{k})^{-1}V_{k}^{T}\mathbf{v}\,d\,z \\
=& V_{k}f(V_{k}^{T}A_{n}V_{k})V_{k}^{T}\mathbf{v}.
\end{split}
\label{eq:generic_krylov_approximation}
\end{equation}%
Different methods for the approximation of matrix functions are obtained for
different choices of the projection spaces $\mathcal{W}_{k}(A_{n},\mathbf{v})
$. Given a set of scalars $\left\{ \sigma _{1},\dots,\sigma _{k-1}\right\}
\subset \overline{\mathbb{C}}$ (the extended complex plane), that are not
eigenvalues of $A_{n}$, let
\begin{equation*}
q_{k-1}(z)=\prod\nolimits_{j=1}^{k-1}(\sigma _{j}-z).
\end{equation*}
The rational Krylov subspace of order $k$ associated with $A_{n}$, $\mathbf{v%
}$ and $q_{k-1}$ is defined by%
\begin{equation*}
\mathcal{Q}_{k}(A_{n},\mathbf{v})=\left[ q_{k-1}(A_{n})\right] ^{-1}\mathcal{%
K}_{k}(A_{n},\mathbf{v}),
\end{equation*}
where%
\begin{equation*}
\mathcal{K}_{k}(A_{n},\mathbf{v})=\func{Span}\{\mathbf{v},A_{n}\mathbf{v}%
,\ldots ,A_{n}^{k-1}\mathbf{v}\}
\end{equation*}
is the standard polynomial Krylov space. By defining the matrices%
\begin{equation*}
C_{j}=\left( \mu _{j}\sigma _{j}A_{n}-I\right) (\sigma _{j}I-A_{n})^{-1},
\end{equation*}
where $\left\{ \mu _{1},\dots, \mu _{k-1}\right\} \subset \overline{\mathbb{C}}$
are such that $\sigma _{j}\neq $ $\mu _{j}^{-2}$, it is known that the
rational Krylov space can also be written as follows (see~\cite%
{guttel2013rational})%
\begin{equation*}
\mathcal{Q}_{k}(A_{n},\mathbf{v})=\func{Span}\{\mathbf{v},C_{1}\mathbf{v}%
,\ldots ,C_{k-1}\cdots C_2 C_{1}\mathbf{v}\}.
\end{equation*}
This general formulation allows to recast most of the classical Krylov
methods in terms of a rational Krylov method with a specific choice of $%
\sigma _{j}$ and $\mu _{j}$. In particular, the standard (polynomial) Krylov
method in which $\mathcal{W}_{k}(A_{n},\mathbf{v})=\mathcal{K}_{k}(A_{n},%
\mathbf{v})$ can be recovered by defining $\mu _{j}=1$ and $\sigma
_{j}=\infty $ for each $j$. The extended Krylov method (see~\cite%
{EKSM1,EKSM2}), in which
\begin{equation*}
\mathcal{W}_{2k-1}(A_{n},\mathbf{v})=\func{Span}\{\mathbf{v}, A_{n}^{-1}%
\mathbf{v}, A_{n}\mathbf{v}, \ldots, A_{n}^{-(k-1)}\mathbf{v}, A_{n}^{k-1}%
\mathbf{v}\},
\end{equation*}
is obtained by setting
\begin{equation*}
(\mu _j,\sigma_j)= \left\lbrace \begin{array}{ll}
(1,\infty), & \mbox{for  $j$ even},\\
(0,0),  & \mbox{for  $j$  odd.}
\end{array}\right.
\end{equation*}
The shift-and-invert (single pole) rational Krylov, see~\cite{moret2004rd,van2006preconditioning}, where%
\begin{equation*}
\mathcal{W}_{k}(A_{n},\mathbf{v})=\func{Span}\{\mathbf{v},(\sigma
I-A_{n})^{-1}\mathbf{v},\ldots ,(\sigma I-A_{n})^{-(k-1)}\mathbf{v}\},
\end{equation*}
is defined by taking $\mu _{j}=0$ and $\sigma _{j}=\sigma $ for each $j$.

As for the use of Krylov methods for the computation of functions of
operators involving $(-\Delta )^{\alpha /2}$, that is for our cases
of interest, we quote here~\cite{yang2011novel} and~\cite{MR3854059}
in which the standard Krylov method and the shift-and-invert
approach are investigated, respectively. On the same line,  the
computation of the matrix square root ($\alpha = 1$) is done by
using  the extended Krylov method in~\cite{MR1616584}, 
see also \cite{MR3069090}, and the rational Lanczos
approximation in~\cite{MR2522957}.

In this work we consider a rational Krylov approach in which $\mu _{j}=0$
for each $j$, and where $-\sigma _{j} = \xi _{j}>0$\ are suitably defined
(see Section~\ref{sec:Optimal_Poles_Selection}). In this view, the method
here presented is a rational Krylov method in which%
\begin{equation} \label{eq:rational_kryl_subs}
\mathcal{W}_{k}(A_{n},\mathbf{v})=\func{Span}\{\mathbf{v},(\xi
_{1}I+A_{n})^{-1}\mathbf{v},\ldots ,(\xi _{k-1}I+A_{n})^{-1}\cdots(\xi
_{1}I+A_{n})^{-1}\mathbf{v}\}.
\end{equation}
Starting from $\mathbf{v}_{1}=\mathbf{v/}\beta $, where $\beta =\Vert
\mathbf{v}\Vert _{2}$, we determine $\mathbf{v}_{j+1}$ by orthogonalizing
the vector%
\begin{equation} \label{lsk}
\mathbf{w}_{j}=(\xi _{j}I+A_{n})^{-1}\mathbf{v}_{j}
\end{equation}
against $\mathbf{v}_{1},\dots, \mathbf{v}_{j}$, followed by normalization.
In this way, a sequence of vectors $\{\mathbf{v}_{j}\}_{j=1}^{k}$ is generated
such that
\begin{align*}
\mathbf{v}_{j}  &=  (\xi_j I + A_n) \sum_{i=1}^{j+1}h_{i,j} \mathbf{v}_{i}, \qquad   \text{  for } j\leq k-1, \\
\mathbf{v}_{k}  &=   (\xi_k I + A_n) \sum_{i=1}^{k} h_{i,k} \mathbf{v}_{i} + (\xi_k I + A_n)  h_{k+1,k} \mathbf{v}_{k+1}
\end{align*}
obtaining the following Arnoldi--like decomposition:
\begin{equation*}
V_k^T A_n V_k = (I -H_k D_k ) H_k^{-1}    - h_{k+1,k} V_k^T   A_n  \mathbf{v}_{k+1}  \mathbf{e}_{k}^{T}  H_k^{-1},
\end{equation*}
where $D_{k}=\func{diag}(\{\xi _{j}^{-1}\}_{j=1}^{k})$ and $H_{k}$ is the
Hessemberg matrix $H_{k}=[h_{i,j}]$. Since
\begin{equation*}
h_{k+1,k} V_k^T   A_n   \mathbf{v}_{k+1}  \mathbf{e}_{k}^{T}  H_k^{-1} =    h_{k+1,k} V_k^T  V_{k+1} H_{k+1} \mathbf{e}_{k+1}    \mathbf{e}_{k}^{T}  H_k^{-1}  =O_k,
\end{equation*}
we find the following expression
for the projected matrix
\begin{equation*}
V_k^T A_n V_k = (I -H_k D_k ) H_k^{-1}.
\end{equation*}
Finally, we approximate $f(A_{n})\mathbf{v}$ as
\begin{equation*}
f(A_{n})\mathbf{v}\approx \beta V_{k}f(V_{k}^{T}A_{n}V_{k})\mathbf{e}_{1},
\end{equation*}
where $\mathbf{e}_{1}=(1,0,\dots,0)^{T}\in \mathbb{R}^{k}.$

As is well known, rational Krylov methods are generally quite fast (in terms of
iterations with respect to the polynomial counterpart) whenever $A_{n}$
represents an unbounded self-adjoint operator. In this view, the most
computational demanding part of these methods is typically due to the
solution of the linear systems~\eqref{lsk}. The computational cost of an
iterative solver for~\eqref{lsk}, whenever $A_{n}$ is large, is highly
problem-dependent and often influences very much the performance of these
algorithms and therefore their competitiveness with respect to polynomial
methods. Anyway, in the examples considered here, we always deal with sparse
and, in particular, banded matrices $A_{n}$. In this sense, we compute the
approximate solution of~\eqref{lsk} by means of sparse direct solvers.
Alternatively, one may consider iterative solvers with preconditioners able
to handle multiple shifts (see e.g.~\cite%
{benzi2003approximate,bertaccini2004efficient,bellavia2011nonsymmetric,bertaccini2016interpolating}%
).

In Section~\ref{sec:Optimal_Poles_Selection} we propose to select
the poles $\xi _{j}$ by exploiting some insights on the functions to
be approximated here, i.e. $f(z)=z^{-\alpha /2}$ and $f(z)=(1+\nu
z^{\alpha /2})^{-1}$, by working with suitable rational
approximations. Clearly, one could look for an automatic selection
of the poles that avoids the use of such information. An example of
application of this strategy is represented by the \texttt{RKFIT}
algorithm~\cite{MR3702872}. It can be used for computing a rational
approximation for a matrix function $f(A_n)$ by solving a rational
least square problem of the form
\begin{equation*}
\text{given }f(A_n),A_n,\mathbf{v},\text{ find
}R_{k+\ell,k}(\cdot)\text{ to minimize } \Vert
f(A_n)\mathbf{v}-R_{k+\ell,k}(A_n)\mathbf{v}\Vert _{2}^{2},
\end{equation*}
where $f$ is a matrix function (or an approximation), and
$R_{k+\ell,k}$ is a rational function of type $(k+\ell,k)$ and $\ell
\geq -k$. Thus, the minimum is taken on the roots of the polynomials
defining $R_{k+\ell,k}$. To effectively use this procedure, one
needs to find some starting poles for $R_{k+\ell,k}(z)$ and the
values of $k$ and $\ell.$

\section{Poles Selection}
\label{sec:Optimal_Poles_Selection}

Our proposal for selecting the poles for the construction of the rational Krylov subspace~\eqref{eq:rational_kryl_subs} relies on the rational
approximation of $z^{-\alpha/2}$ proposed in~\cite{aceto2017efficient,aceto2017rational,2018arXiv180710086A}. In particular, following \cite[eq. (5)]{aceto2017efficient}, we get
\begin{equation*}
z^{-\alpha/2} \approx \sum_{j=1}^{k} \frac{2\sin(\frac{\alpha}{2}\pi) \tau^{1-\alpha/2}}{\pi} \frac{\omega_j}{1+\theta_j} \left(\frac{\tau(1-\theta_j)}{1+\theta_j} +z\right)^{-1}  \triangleq R_{k-1,k}\left( z \right),
\end{equation*}
where $\omega_j$ and $\theta_j$ are, respectively, the weights and
nodes of the Gauss--Jacobi quadrature formula with weight function
$(1-x)^{-\frac{\alpha}{2}} (1+x)^{\frac{\alpha}{2}-1}$ and $\tau$ is
a positive real parameter  that should be defined suitably.  In
addition, denoting by $\zeta_r$ the $r$th zero of the Jacobi
polynomial
${\mathcal{P}}_{k-1}^{({\alpha}/{2},1-{\alpha}/{2})}\left(z\right)$
and setting
\begin{eqnarray}
\epsilon_r &=& \tau \frac{1- \zeta_r}{1+\zeta_r}, \quad r=1,2,\dots,k-1, \label{eq:epsrad} \\
\eta_j &=& \frac{\tau(1-\theta_j)}{1+\theta_j}, \quad j=1,2,\dots,k,  \label{eq:etarad}
\end{eqnarray}
from~\cite[Proposition~1]{aceto2017efficient} we can express $R_{k-1,k}(z)$ as the rational function
\begin{equation*}
R_{k-1,k}(z) = \frac{p_{k-1}(z)}{q_k(z)} =\frac{\chi \prod_{r=1}^{k-1}(z+\epsilon_r)}{\prod_{j=1}^{k} (z+\eta_j)},
\end{equation*}
where
\begin{equation*}
 \quad \chi =
\frac{\eta_k}{\tau^{\alpha/2}}
\frac{\binom{k+\nicefrac{\alpha}{2}-1}{k-1}}{\binom{k-\nicefrac{\alpha}{2}}{k}}
\prod_{j=1}^{k-1} \frac{\eta_j}{\epsilon_j}.
\end{equation*}

It is worth noting that $\{\theta_j\}_{j=1}^k$ are the zeros of the Jacobi polynomial
${\mathcal{P}}_{k}^{(-{\alpha}/{2},{\alpha}/{2}-1)}\left(z\right).$ In this context, this implies that  the values $\eta_j$
are all real and simple and therefore the roots of $q_k(z)$ as well.

Given a positive definite matrix $A_n$ with spectrum $\sigma(A_n)
\subseteq [\lambda_{\min}, \lambda_{\max}],$ a specific rational
approximation for the matrix function $f(A_n)=A_n^{-\alpha/2}$ can
be deduced from the above, provided that the parameter $\tau$ is
selected. To this purpose, setting
\begin{equation*}
{\tilde \tau}_k   \triangleq \lambda_{\text{min}} \left( \frac{{\alpha}/{2}}{2
    k e}  \right)^2 \exp\left(2 W \left( \frac{4k^2
    e}{({\alpha}/{2})^2}\right)\right)
\end{equation*}
with $W(\cdot)$ denoting the Lambert-W function, in~\cite[Propositions~3.4, 4.1]{2018arXiv180710086A} it has been shown that
 for the matrix function $A_n^{-\alpha/2}$ a reliable value of $\tau$ is given by
\begin{equation*}
\tau = \tau_k := \left\lbrace\begin{array}{ll}
  {\tilde \tau}_k , & \text{ if } k \le \overline{k},\\
\displaystyle \left( {\tilde \sigma}_k + \sqrt{{\tilde \sigma}_k^2 + (\lambda_{\text{min}}\lambda_{\text{max}})^{\nicefrac{1}{2}}} \right)^2, & \text{ if } k> \overline{k},\\
\end{array}\right.
\end{equation*}
where
\begin{equation*}
{\tilde \sigma}_k  \triangleq  - \frac{\nicefrac{\alpha}{2}}{8 k}  \ln\left(\frac{\lambda_{\text{max}}}{\lambda_{\text{min}}}\right)   \lambda_{\text{max}}^{\nicefrac{1}{2}},
\end{equation*}
and
\begin{equation*}
\overline{k} = \frac{(\nicefrac{\alpha}{2})^2}{8}  \left(\frac{\lambda_{\text{max}}}{\lambda_{\text{min}}}\right)^{\nicefrac{1}{2}} \left[  \ln\left(\frac{\lambda_{\text{max}}}{\lambda_{\text{min}}}\right)  +2\right].
\end{equation*}
This choice is the result of an analysis on the Pad\'{e}-type approximations of $z^{-\alpha/2}.$
Now, by using the above arguments, we can get an approximation also for $f(z)=(1+\nu z^{\nicefrac{\alpha}{2}})^{-1}.$  Indeed, we can write
\begin{equation*}
\begin{split}
(1+\nu z^{\nicefrac{\alpha}{2}})^{-1} =&\frac{1}{1+\nu (z^{-{\nicefrac{\alpha}{2}}})^{-1}} \approx  \frac{1}{1+\nu (R_{k-1,k}(z))^{-1}}  =  \frac{1}{1+\nu   \frac{q_k(z)}{p_{k-1}(z)}}\\
=& \frac{p_{k-1}(z)}{p_{k-1}(z)+\nu q_k(z)} := \frac{p_{k-1}(z)}{{\tilde q}_k(z)}.
\end{split}
\end{equation*}
For the roots of $\tilde{q}_k(z)$ we observe the following useful
result.
\begin{proposition}\label{prop:realroots}
    All the roots of the polynomial $\tilde{q}_k(z)$ are real and
    simple.
\end{proposition}
To prove this we need two auxiliary results. First, there is an interlacing of the zeros of
Jacobi polynomials of different orders and weights.
\begin{theorem} \cite[Theorem 2.3]{INTERLACING} \label{thm:interlacing}
    Let $\beta, \gamma > -1.$ For any $t, s \in [0, 2],$ let
    \begin{itemize}
        \item $-1 < \theta_1 < \theta_2 < \ldots < \theta_n < 1$ be the zeros of  ${\mathcal{P}}^{(\beta,\gamma)}_n,$ and
        \item $-1 < \zeta_1 < \zeta_2 < \ldots < \zeta_{n-1} < 1$ be the zeros of ${\mathcal{P}}^{(\beta+t,\gamma+s)}_{n-1}$.
    \end{itemize}
    Then
    \begin{equation*}
    -1 < \theta_1 < \zeta_1 < \theta_2 < \zeta_2 < \ldots < \theta_{n-1} < \zeta_{n-1} < \theta_n < 1.
    \end{equation*}
\end{theorem}
\begin{theorem}\cite[Theorem 8]{REALROOT}\label{thm:realroot}
    Let $p$ and $q$ be real polynomials. Then $p$ and $q$ have strictly
    alternating roots if and only if all polynomials in the space
    \begin{equation*}
    \{ c_1 \, p + c_2 \, q \,:\, c_i \in \mathbb{R}, i=1,2 \}
    \end{equation*}
    have real and simple roots.
\end{theorem}

\begin{proof}[Proof of Proposition~\ref{prop:realroots}]
    To apply Theorem~\ref{thm:realroot}, we need to show that the roots
    of the polynomials $p_{k-1}(z)$ and $q_{k}(z)$ are strictly
    interlaced. By construction we know that the roots of $p_{k-1}(z)$
    are the real values $\{-\epsilon_r\}_{r=1}^{k-1}$
    in~\eqref{eq:epsrad}, while the roots of $q_k(z)$ are the real
    values $\{-\eta_j\}_{j=1}^{k}$ in~\eqref{eq:etarad}. By
    Theorem~\ref{thm:interlacing} applied to the Jacobi polynomials
    ${\mathcal{P}}_{k}^{(-{\alpha}/{2},{\alpha}/{2}-1)}\left(z \right)$
    and ${\mathcal{P}}_{k-1}^{({\alpha}/{2} ,1-{\alpha}/{2}
        )}\left(z\right)$ we have
    \begin{equation}\label{eq:interlacing}
    -1 < \theta_1 < \zeta_1 < \theta_2 < \zeta_2 < \ldots < \theta_{k-1} < \zeta_{k-1} < \theta_k < 1.
    \end{equation}
    Then, for the interlacing between the $\{-\epsilon_r\}_{r=1}^{k-1}$
    and the $\{-\eta_j\}_{j=1}^{k}$, we need to prove that   $ -\eta_1 < -\epsilon_1 < -\eta_2 < -\epsilon_2 < \ldots < -\eta_{k-1} <
   -\epsilon_{k-1} < -\eta_k$  or, equivalently,
    \begin{equation*}
    \begin{gathered}
    \eta_1 > \epsilon_1 > \eta_2 > \epsilon_2 > \ldots > \eta_{k-1} >
   \epsilon_{k-1} > \eta_k.
    \end{gathered}
    \end{equation*}
Now, using  \eqref{eq:epsrad} and \eqref{eq:etarad}  for $i=1,\dots,k-1,$  we have
    \begin{equation*}
 \epsilon_i > \eta_{i+1}   \quad \Leftrightarrow  \quad  \frac{\tau  \left(1-\zeta _i\right)}{1+\zeta _i}>\frac{\tau
        \left(1-\theta _{i+1}\right)}{1+\theta _{i+1}}.
    \end{equation*}
Since $\tau>0,$ $(1+\zeta_i)>0$ and $(1+\theta _{i+1} )>0,$   the previous inequalities are satisfied if and only if $\zeta _i< \theta_{i+1},$ which is true by \eqref{eq:interlacing}.
Similar arguments lead to verify that $\eta_i >  \varepsilon_i.$ By using Theorem~\ref{thm:realroot}, the proof is complete.
\end{proof}
In addition, since  all the coefficients of ${\tilde q}_k(z)$ are strictly positive by construction, according to the Descartes' rule of signs, we are also sure that all its roots are negative.

In consideration of all the above arguments, the poles that we use  for the computation of the proposed rational Krylov methods are positive, real and simple.
In fact, when $f(A_n)=A_n^{-\alpha/2},$ in \eqref{eq:rational_kryl_subs} we take as poles the opposite of the roots of $q_k(z)$ (i.e., $\xi_j = \eta_j$), while, when $f(A_n)=(I+\nu A_n^{\nicefrac{\alpha}{2}})^{-1},$ we take as poles the  opposite of the roots of ${\tilde q}_k(z).$  Thus, one of the main features of our proposal is represented by the possibility to work only with real arithmetic.  From now on, we refer to  the rational Krylov methods based on these poles as \emph{Krylov Jacobi}.

Just to provide an example of computed poles, in
Figure~\ref{fig:poles_time_stepping} we focus on the matrix $A_n$
given by the centered differences discretization of the 1D Laplacian
with $n =2^{12}$ and $\alpha=1.2.$ In this case,
for $k=10,20,30,$ we computed  the roots
$\{\theta_j\}_{j=1}^{k}$ and $\{\zeta_r\}_{r=1}^{k-1}$  by the
\texttt{JACOBI\_POLYNOMIAL} Matlab code; see~\cite{jacobi}. On the
other hand, we used the package MPSolve to compute the roots of
$\tilde q_k(z)$; see~\cite{MPSOLVE,MPSOLVE2}.

\begin{figure}[htbp]
    \centering
    \includegraphics[width=\columnwidth]{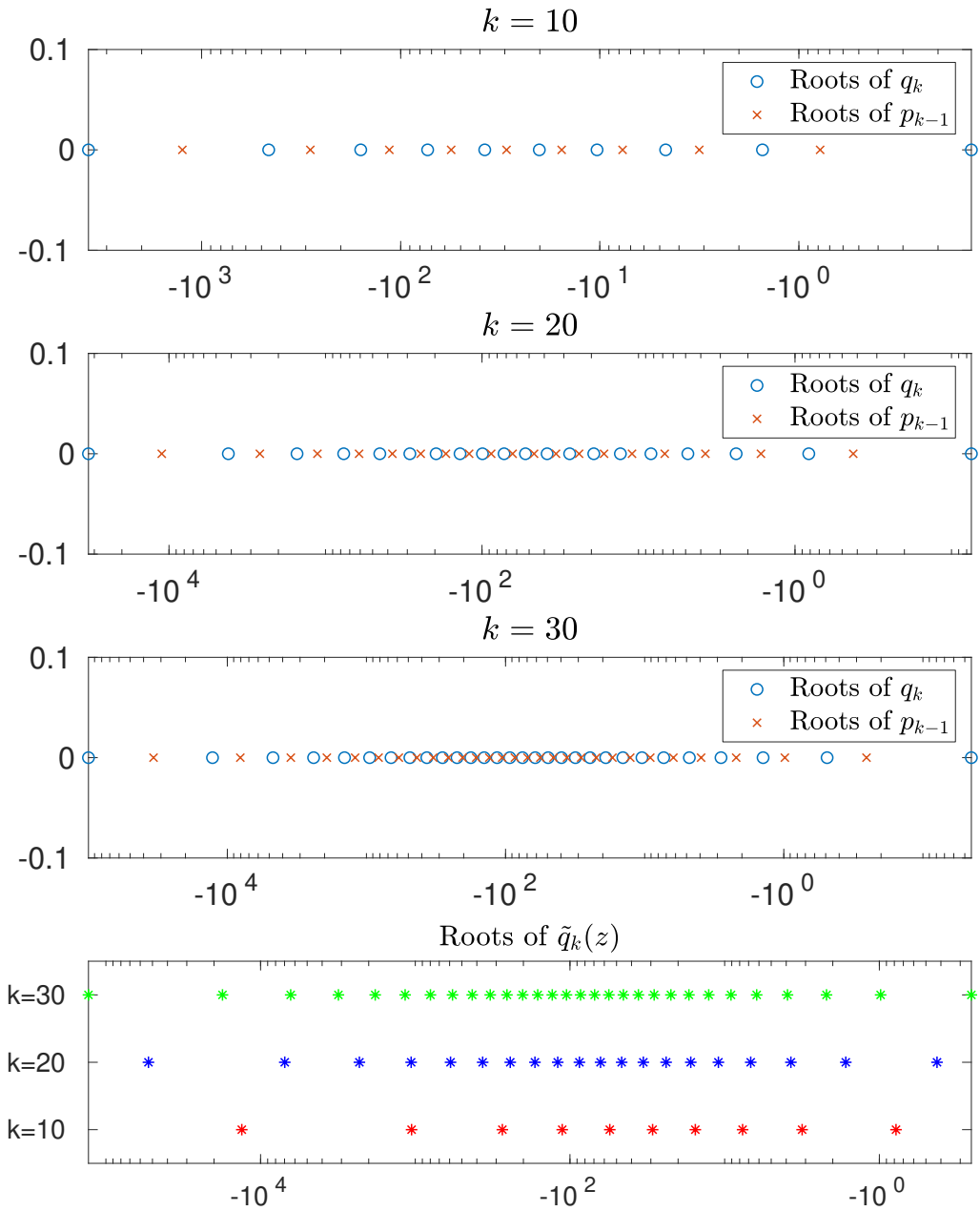}
    \caption{Poles for the computation of the function
    $f(z) = (1+\nu z^{\nicefrac{\alpha}{2}})^{-1}$ for the matrix $A_n$ given by the centered differences
    discretization of the 1D Laplacian with $n =2^{12}$, $\alpha=1.2$, and $\nu = 1/(n+1)$.}
    \label{fig:poles_time_stepping}
\end{figure}

\section{Solving the model problems by MTT using Krylov methods}
\label{sec:Numerical_Solution_of_the_Model_Problems}

Computing efficiently functions of matrix-vector products by Krylov
methods is beneficial for many applications. However, in order to
provide reasonable comparisons, here we consider one class of
mathematical tools which is quite wide and important because it
forms the ground of many numerical models:
(evolutionary) fractional partial differential
equations.

As a case study, let us consider the $d$--dimensional Laplacian
operator
\begin{equation}\label{eq:laplacianoperator}
\Delta = \sum_{i=1}^{d} \frac{\partial^2}{\partial x_i^2},
\end{equation}
on the domain $\Omega$ subject to Dirichlet (or Neumann) boundary conditions.
The \emph{Matrix Transfer Technique} (MTT) introduced in \cite{ilic2005numerical,ilic2006numerical} considers approximations  to the fractional Laplacian of the form
\begin{equation*}
(-\Delta)^{\frac{\alpha}{2}} \approx A_n^{\nicefrac{\alpha}{2}},
\end{equation*}
where $A_n$ is a suitable matrix approximation of the standard Laplacian. In this work, we consider the following discretizations:
\begin{itemize}
\item finite differences (FD);
\item finite elements methods (FEM) with first order Lagrangian elements on a triangular
tessellation;
\item cell--centered finite volume (FV) methods with a piecewise linear basis,~\cite{MR1897941}.
\end{itemize}
For FEM and FV we get $A_n = M_n^{-1}K_n,$ where $K_n$ denotes the
stiffness matrix, while $M_n$ is the mass matrix for the FEM or the
diagonal matrix of the cell volumes for FV.
In these cases, even if $A_n$ is generally not symmetric, it is
similar to a positive definite matrix. In fact,
$M_n^{\frac{1}{2}}A_n M_n^{-\frac{1}{2}}= M_n^{-\frac{1}{2}} K_n M_n^{-\frac{1}{2}}.$
Whenever the mass matrix $M_n$ is not trivial, i.e.
$A_n $ is not a multiple of $K_n,$ we should solve linear systems with
coefficient matrices $(\xi_{j}I + M_n^{-1}K_n),$  by means of a direct method. This can be
inappropriate because it would require assembling the matrix $A_n =
M_n^{-1}K_n$ that can be dense. Therefore, instead of computing the
basis of the rational Krylov space with $A_n,$ we generate it working with
$\hat{A}_n = \xi_j M_n + K_n$. This is completely analogous to what is usually done for generalized
eigenvalues for the pencil $(K_n,M_n)$ with rational Krylov methods,
see, e.g., \cite{RUHE1,RUHE2}.
Then, with the attained basis, we approximate the
matrix--vector product $f(A_n)\mathbf{v}$ as in \eqref{eq:generic_krylov_approximation}.

Finally, we stress that direct solvers for sparse linear systems can
benefit by the use of permutations to reduce fill--in. This can be
mostly useful when unstructured matrices are used for either FEM or
FV discretizations.

\subsection{Integration in time}
\label{sec:time_integrations}

If we consider the solution of the problem~\eqref{eq:time_dependent}
semidiscretized with respect to the space variables, we get the
initial value problem
\begin{equation}
\label{eq:ivp} \left\lbrace\begin{array}{ll}
\mathbf{y}'(t) = - A_n^{\frac{\alpha}{2}} \mathbf{y}(t) + \mathbf{s}(t), & t \in (t_0,T],  \\
\mathbf{y}(t_0) = \mathbf{y}_0,
\end{array}\right.
\end{equation}
where ${\bf y}(t): \mathbb{R} \rightarrow \mathbb{R}^{n}$ and ${\bf
y_0} \in \mathbb{R}^n$ (with respect to problem
\eqref{eq:time_dependent}, to simplify notation, here we set $\mu=1$
and $s(\mathbf{x},u,t)=s(\mathbf{x},t)$). To
integrate~\eqref{eq:ivp} in time, we apply a
\emph{Linear Multistep Method} {(LMM)} over an equispaced grid
$t_{j}= t_0 + j \delta_t, j=0,1,\ldots,n_t,$ with $\delta_t =
(T-t_0)/n_t$ the selected stepsize. Setting
$\mathbf{g}(\mathbf{y},t)= - A_n^{\frac{\alpha}{2}} \mathbf{y}(t) +
\mathbf{s}(t),$ we get
\begin{equation}\label{eq:lmf}
\sum_{j=0}^{\ell}\alpha_{j}\mathbf{y}^{(m+j)} = \delta_t
\sum_{j=0}^{\ell} \beta_j \mathbf{g}^{(m+j)}, \qquad m=0,1,\dots,n_t-\ell,
\end{equation}
with
\begin{equation*}
\mathbf{y}^{(m+j)} \approx \mathbf{y}(t_{m+j}), \quad
\mathbf{g}^{(m+j)} = \mathbf{g}( \mathbf{y}^{(m+j)} ,t_{m+j}).
\end{equation*}
Since the fractional Laplacian operator is unbounded and give
stiffness , implicit schemes are preferable to avoid possible severe
stepsize restrictions in order to satisfy stability requirements;
see, e.g., \cite{Lambert}. Using implicit schemes, we need to solve
linear systems of the form
\begin{equation*}
\left( \alpha_\ell I + \delta_t \beta_{\ell}
A_{n}^{\frac{\alpha}{2}} \right) \mathbf{y}^{(m+\ell)} = \delta_t
\sum_{j=0}^{\ell-1} \beta_{j} \textbf{g}^{(m+j)}
-\sum_{j=0}^{\ell-1} \alpha_j \mathbf{y}^{(m+j)},\, m=0,1,\dots,n_t-\ell,
\end{equation*}
that can be recast in the computation of a function of matrix times vector, that is,
\begin{equation*}
\mathbf{y}^{(m+\ell)} = f(A_n) \tilde{\mathbf{y}}^{(m)} , \qquad m=0,1,\dots,n_t-\ell,
\end{equation*}
where
\begin{equation}\label{eq:time-marching-function}
f(z) = \left( 1 + \delta_t \frac{\beta_{\ell}}{\alpha_\ell }
z^{\frac{\alpha}{2}} \right)^{-1}
\end{equation}
and
\begin{equation*}
\tilde{\mathbf{y}}^{(m)}=  \frac{1}{\alpha_\ell} \left(\delta_t \sum_{j=0}^{\ell-1}
\beta_{j} \textbf{g}^{(m+j)} -\sum_{j=0}^{\ell-1} \alpha_j \mathbf{y}^{(m+j)} \right).
\end{equation*}
By using similar arguments, we can also deal with semi--linear
problems having non--linear forcing term
\begin{equation}
\label{eq:ivp_semilinear} \left\lbrace\begin{array}{ll}
\mathbf{y}'(t) = -\mu A_n^{\frac{\alpha}{2}} \mathbf{y}(t) + \mathbf{s}(t,\mathbf{y}(t)), & t \in (t_0,T],\\
\mathbf{y}(t_0) = \mathbf{y}_0.
\end{array}\right.
\end{equation}
Assuming that the function $\mathbf{s}$ is not
responsible for further stiffness, we can use
\emph{implicit--explicit methods} (IMEX), similarly
to~\cite{aceto2017efficient}. A generic $\ell$-step IMEX method for
\eqref{eq:ivp_semilinear} can be written as
\begin{equation}\label{eq:IMEX}
\sum_{j=0}^{\ell}\alpha_{j}\mathbf{y}^{(m+j)} = -\delta_t \mu
\sum_{j=0}^{\ell} \beta_j   A_n^{\frac{\alpha}{2}}
\mathbf{y}^{(m+j)} + \delta_t \sum_{j=0}^{\ell-1} \gamma_j
\mathbf{s}(t_{m+j},\mathbf{y}^{(m+j)}).
\end{equation}
For further details on IMEX methods see, e.g.,
~\cite{MR1335656,MR592157}. Therefore, we need again to compute
matrix functions  of the form \eqref{eq:time-marching-function}.

\section{Numerical Experiments}
\label{sec:Numerical_Experiments}

The examples in this section are collected in two groups. In
Section~\ref{sec:Stationary_Problems} we deal with the discrete
version of the steady state problem~\eqref{eq:steady_state}, while
in Section~\ref{sec:Time_Dependent_Problem} the time dependent
problem~\eqref{eq:time_dependent}.

All the numerical experiments are performed on a laptop running Linux with 8~Gb memory and CPU
Intel\textsuperscript{\textregistered} Core\texttrademark\,
i7-4710HQ CPU with clock 2.50~GHz. The codes are written and executed
in MATLAB R2018a. The following external codes are used in our
routines:
\begin{itemize}
    \item MPSolve package for the computation of the polynomial zeros; see~\cite{MPSOLVE},
    \item \lstinline[style=Matlab-bw]|rat_krylov| for the construction of the basis of the rational Kyrlov methods; see~\cite{guttel2013rational},
    \item \lstinline[style=Matlab-bw]|EKS| class for the construction of the basis of the extended Krylov method; see~\cite{EKSM2},
    \item FENICS library~\cite{AlnaesBlechta2015a} v.2018.1 to assemble the finite elements matrices.
    \item FiPy library~\cite{FIPY} v.3.1.3-dev2-g11937196 to assemble the finite volume matrices.
\end{itemize}
The auxiliary linear systems are solved here
by Matlab's standard backslash.
Whenever is feasible, the reference solution $\mathbf{u}^*$  for the
various problem is computed directly by the Schur--Parlett algorithm. We denote by $\varepsilon = \|\mathbf{u}^* -
\mathbf{u}\|_2/\|\mathbf{u}^*\|_2$ the relative error.
 All the timings are measured in
seconds averaged on one hundred runs. For the definition of the poles, the interval $[\lambda_{\min},\lambda_{\max}]$
is always assumed explicitly known.

\subsection{Stationary problems}
\label{sec:Stationary_Problems}

We report the results for the steady state
problem~\eqref{eq:steady_state} discretized using second order
centered differences for the 1D and 2D Laplacian on the domains
$[0,1]$ and $[0,1]^2$, respectively. The right--hand side terms for
the two problems are $f(x) = \sin(\pi x)$, and $f(x,y) = \sin(\pi
x)\sin(\pi y)$. We observe the expected behavior: rational Krylov
Jacobi algorithms outperform (in term of
iterations) the polynomial Krylov; see
Figure~\ref{fig:stationary_problem_1D_finite_differences} for the
behavior for the one dimensional problem, and
Figure~\ref{fig:stationary_problem_2D_finite_differences} for the
two dimensional. In both cases the shift parameter $\sigma$ for the
shift--and--invert method is computed as $\sigma =
\sqrt{\lambda_{\min}(A_n)\lambda_{\max}(A_n)}$, cf.
\cite{MR2522957}.
\begin{figure}[htbp]
    \centering
    \includegraphics[width=0.80\columnwidth]{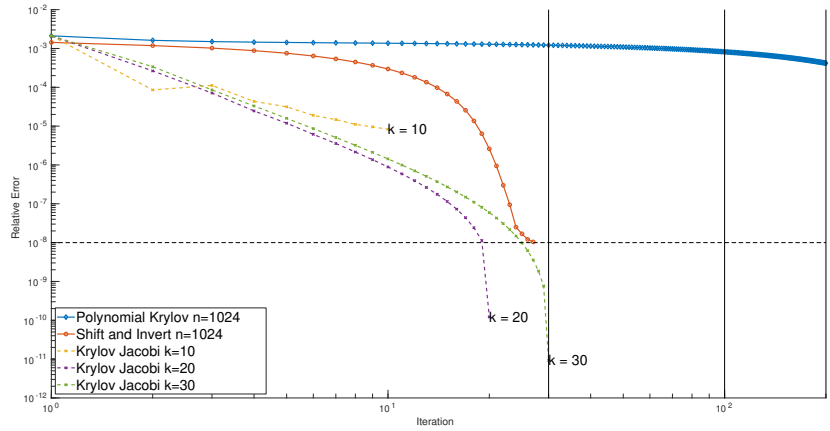}
    \caption{Comparisons of the convergence curve in term of the relative error
    $\varepsilon$ with the reference solution for the 1D stationary problem
    discretized by finite differences for $\alpha = 1.2, k = 10,20,30$. The behavior for other values of $\alpha$ is similar.}\label{fig:stationary_problem_1D_finite_differences}
\end{figure}
\begin{figure}[htbp]
    \centering
    \includegraphics[width=0.80\columnwidth]{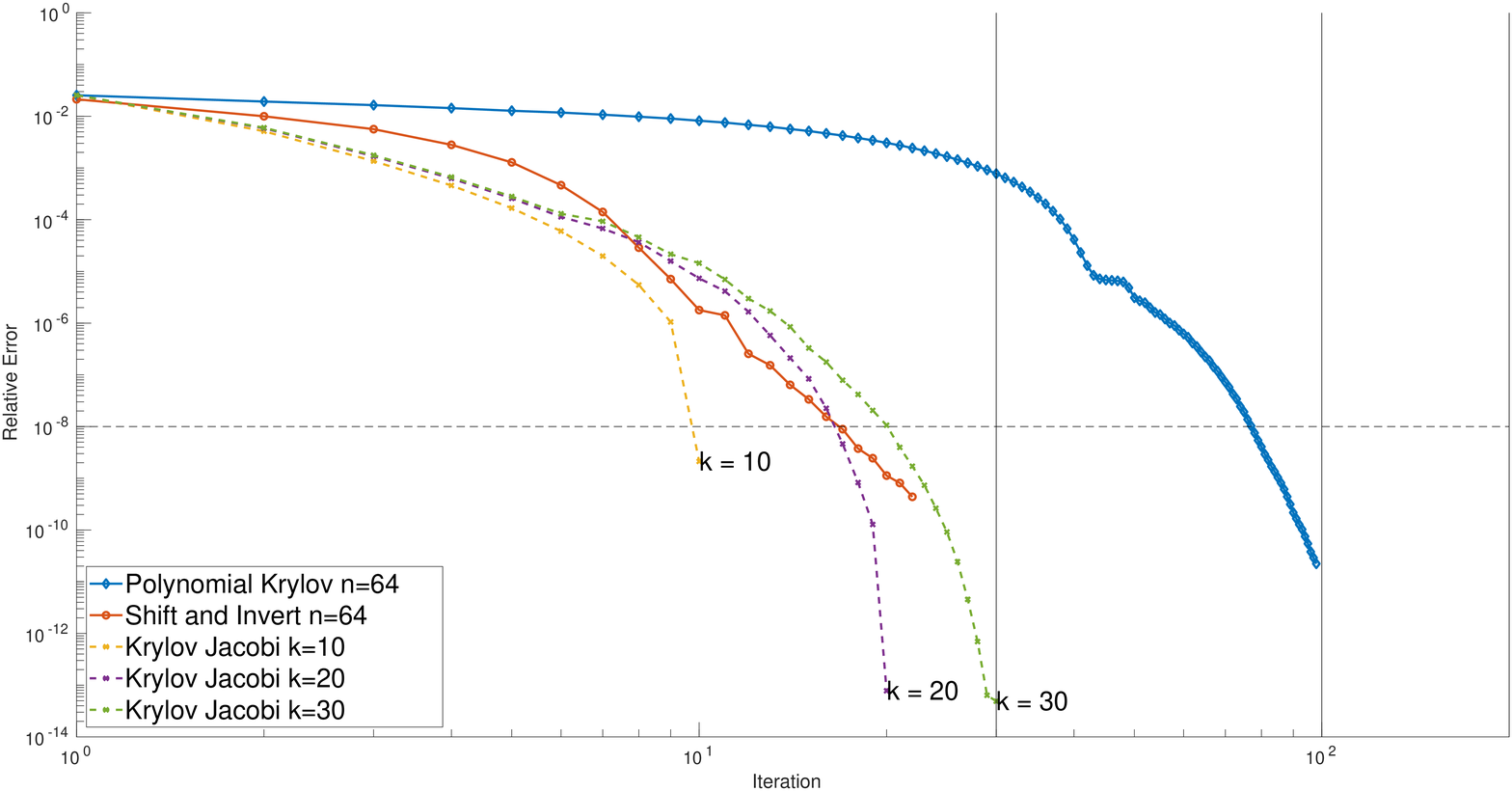}
    \caption{Comparisons of the convergence curve in term of the relative error
    $\varepsilon$ with the reference
    solution for the 2D stationary problems discretized by finite differences for $\alpha = 1.5, k = 10,20,30$. The behavior for other values of $\alpha$ is similar.}
    \label{fig:stationary_problem_2D_finite_differences}
\end{figure}

\subsection{Time--dependent problems}
\label{sec:Time_Dependent_Problem}

In this section, we focus on the problem~\eqref{eq:time_dependent}. First, we set  $s = 0,$  $\mu=1,$ $T=1,$
\begin{equation}\label{eq:f=0_coeff}
u(x,y,0) = x^2 y^2 (1-x)(1-y),
\end{equation}
on the unit square $\Omega = [0,1]^2$ and consider
Dirichlet boundary conditions.  We refer to this
test problem as P1. By using finite differences (FD), finite
volumes (FV) and finite elements methods (FEM) we discretize this
problem and test on it the underlying rational Krylov methods. Then,
we consider the problem~\eqref{eq:time_dependent} with $s$ a nonzero
polynomial function of the solution, producing the \emph{fractional
Allen--Cahn} equation, that is
\begin{equation}\label{eq:allen-cahn}
\left\lbrace\begin{array}{ll}
\displaystyle\frac{\partial u}{\partial t} + \mu(-\Delta)^{\frac{\alpha}{2}} u = - (u^3 - u), & (\mathbf{x},t) \in \Omega\times[t_0,T],\\
u(\mathbf{x},t_0) = u_0(\mathbf{x}), & \mathbf{x} \in \Omega,\\
\partial_n u(\mathbf{x},t) = 0, & {(\mathbf{x},t)} \in \partial \Omega \times (t_0,T].
\end{array}\right.
\end{equation}
Here $\mu > 0$ is a small parameter defining the thickness
of the interface separating the different phases and $u_0 \in
\mathbb{L}^2(\Omega).$
In our numerical experiments we choose $\Omega=[0,1]^2, [t_0,T]=[0,4],$ and
\begin{equation}\label{eq:initial_data_alle_cahn}
u_0(\mathbf{x}) \equiv  u_0(x,y) = 0.25\sin(2\pi x)\sin(2\pi y).
\end{equation}
Refer to this test problem as P2. We apply finite volumes and
finite elements methods to discretize this problem. In both cases
the obtained matrices $A_n$ are positive definite, and thus the
analysis in Section~\ref{sec:Optimal_Poles_Selection} applies
straightforwardly.

\paragraph{P1 by using FD}
We consider the five--point stencil of the FD discretization for the
Laplacian operator on the grid with $n_x = n_y = n_t = 2^6$ points.
The method for marching in time is the \emph{implicit Euler}. Therefore,
as discussed in Section
\ref{sec:Numerical_Solution_of_the_Model_Problems}, in order to
advance in time, we need to compute the matrix function
\begin{equation*}
    f(A_n) = (I + c\,\delta_t A_n^{\alpha/2})^{-1},
\end{equation*}
where $c$ is a constant. In Figure~\ref{fig:f0_FD} the relative
error and execution time comparison for the first iterate of the
method is reported. The poles for the RKFIT algorithm are computed
on the complete matrix function on a reduced size grid ($n_x = n_y =
2^4$ and $n_t = 2^6$) from the initial guess $\{\sigma_j =
+\infty\}_{j=1}^{k}$. Note that the RKFIT can compute only $k = 10$
different nodes. We recall also
that for the extended Krylov algorithm the number $k$
represent the size of the Krylov subspaces, thus the number of
linear auxiliary linear systems solved is $\lfloor {k}/{2} \rfloor$.
This explain the behavior with respect to time: the Krylov space of
the same dimension for the extended Krylov costs roughly a half of
the rational subspace of the same size.
\begin{figure}[htbp]
    \centering
    \subfloat[$\alpha=1.2$]{
    \includegraphics[width=0.32\columnwidth]{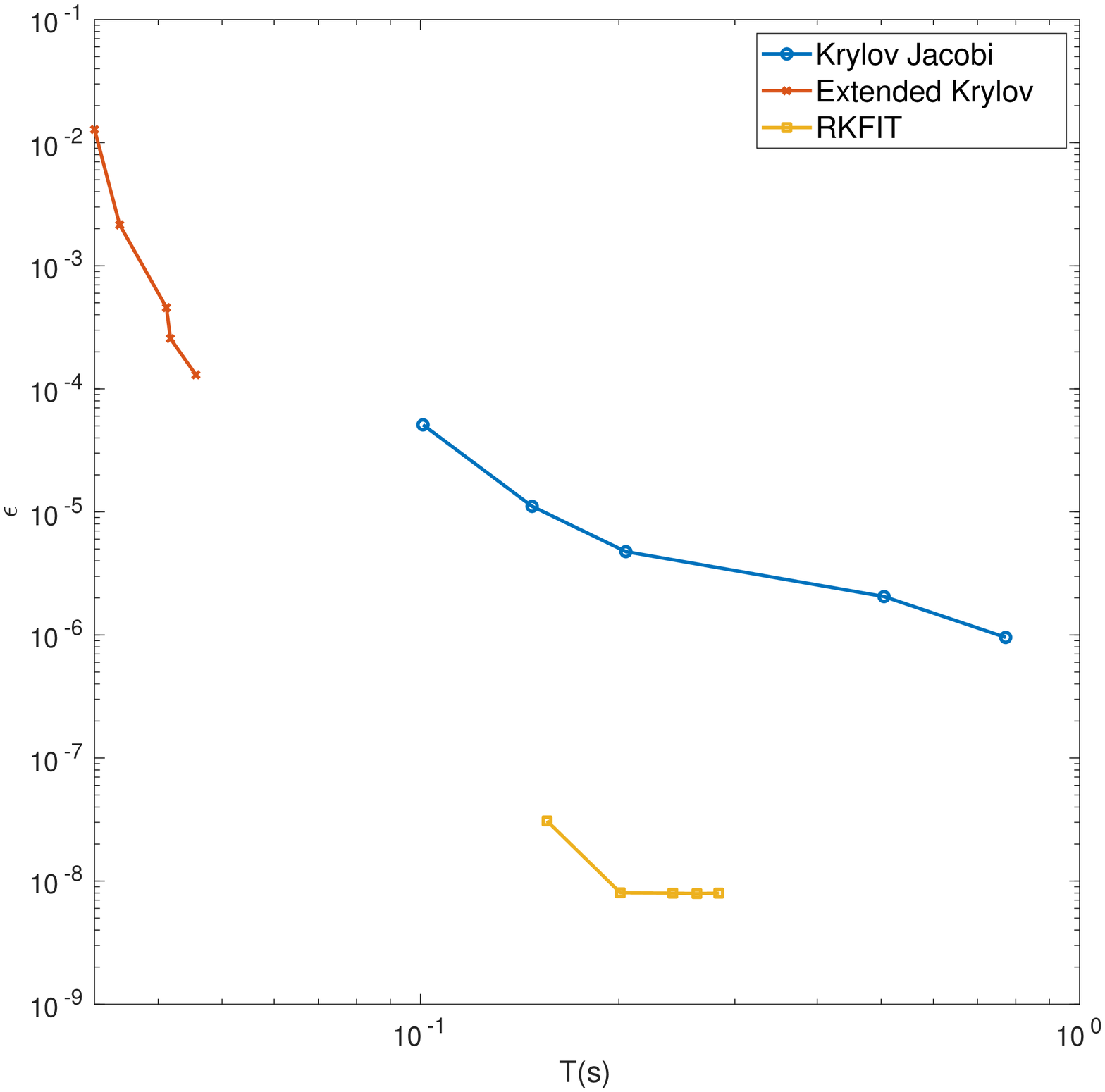}
    }
    \subfloat[$\alpha=1.5$]{
    \includegraphics[width=0.32\columnwidth]{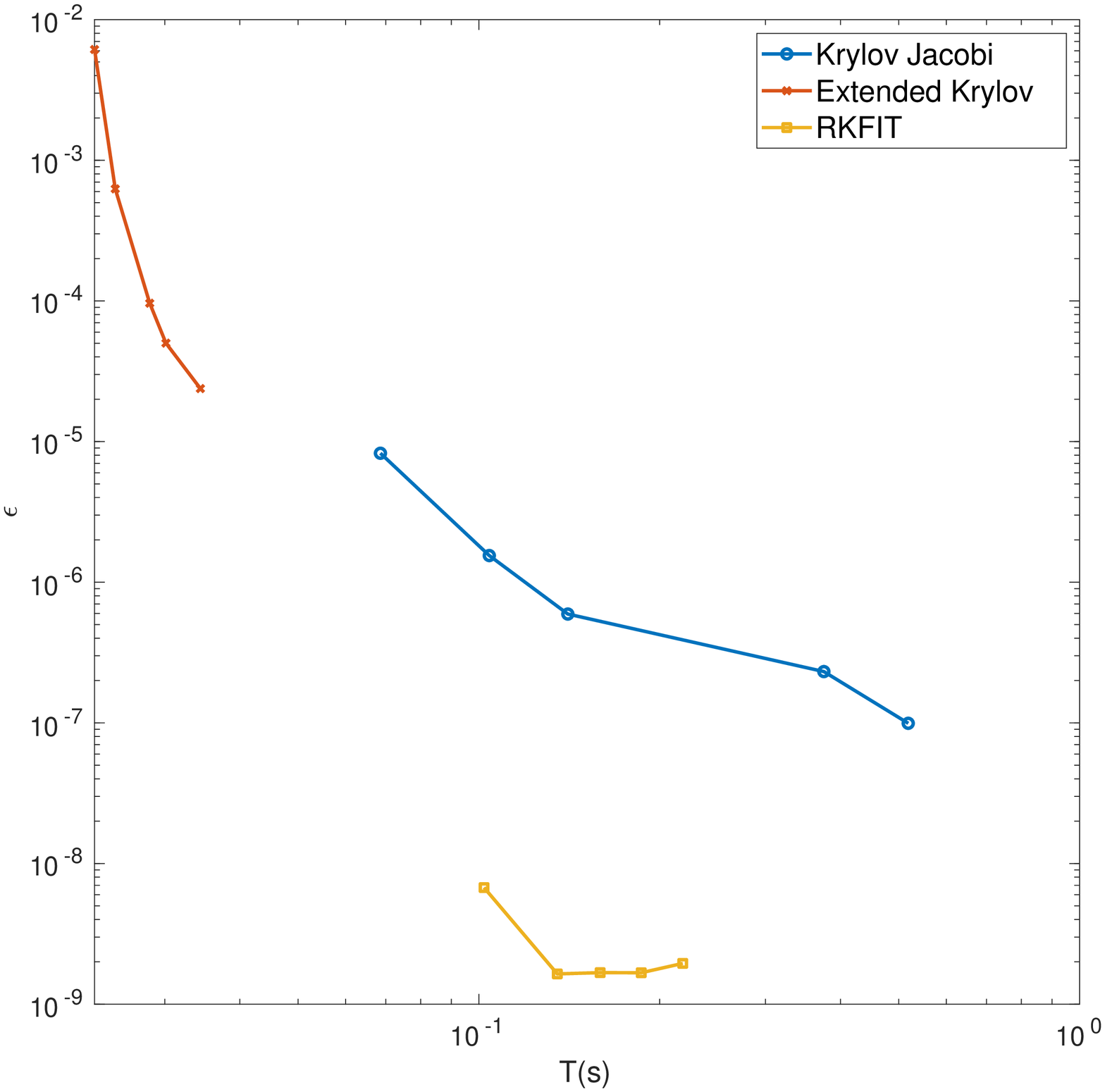}
    }
    \subfloat[$\alpha=1.8$]{
    \includegraphics[width=0.32\columnwidth]{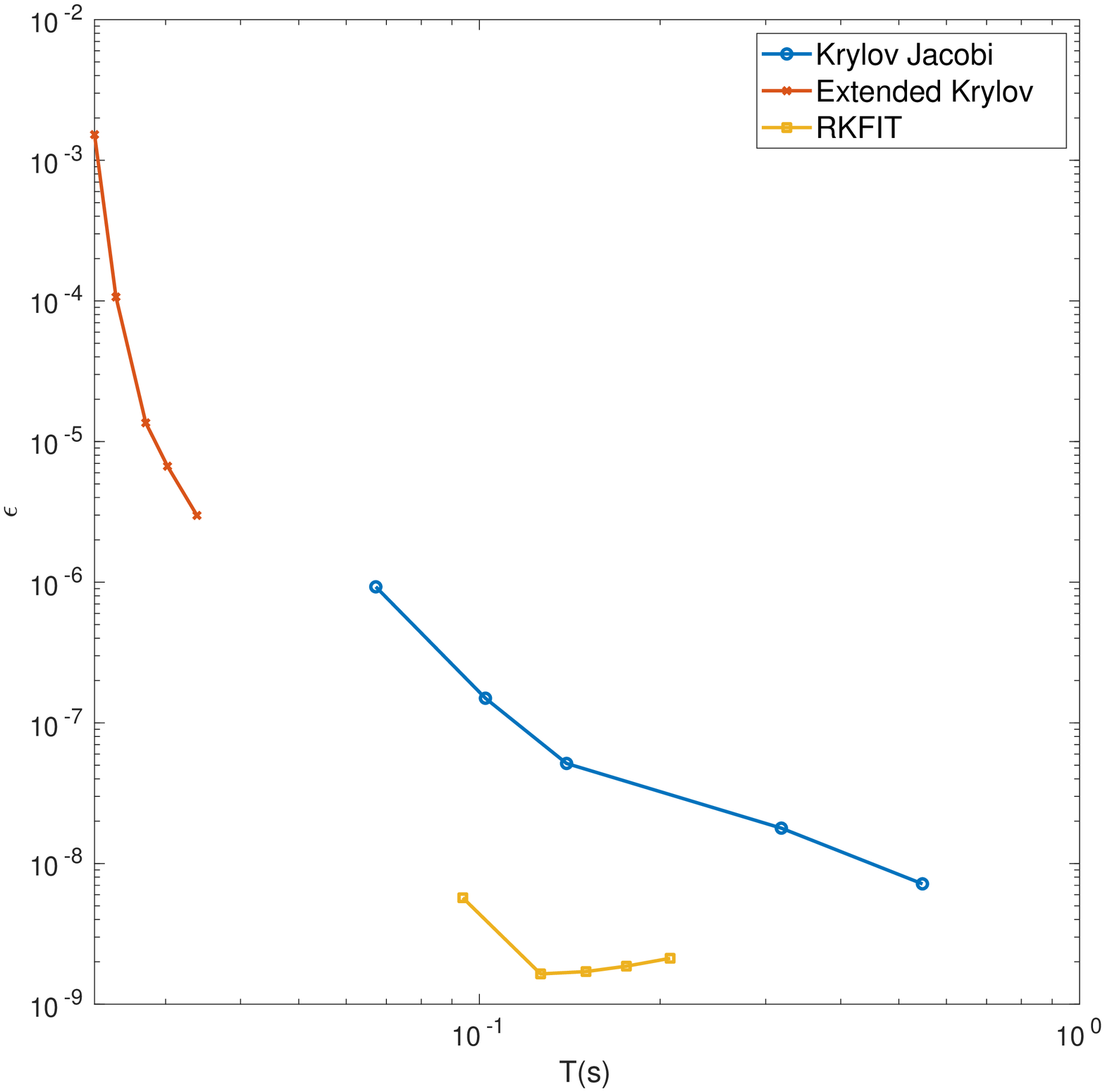}}
    \caption{P1-- FD: The nodes in the time--error graph corresponds to $k =  10,15,20,25,30$ poles.}\label{fig:f0_FD}
\end{figure}

\paragraph{P1 by using FV}
We consider the cell--centered FV discretization of the
problem~\eqref{eq:time_dependent} on the same grid used for
finite differences. In particular, $n_x = n_y = n_t = 2^6$.
Results in Figure~\ref{fig:f0_FV} are comparable with those of the
FD discretization. The RKFIT algorithm produces less than the
requested number of nodes and at a greater computational effort. By
fixing the size of the Krylov space, we observe that the
Krylov Jacobi method  is more accurate than the
extended Krylov method which, on the other hand, keeps the same
ratio with respect to the achieved timings.
\begin{figure}[htbp]
    \centering
    \subfloat[$\alpha=1.2$]{
        \includegraphics[width=0.32\columnwidth]{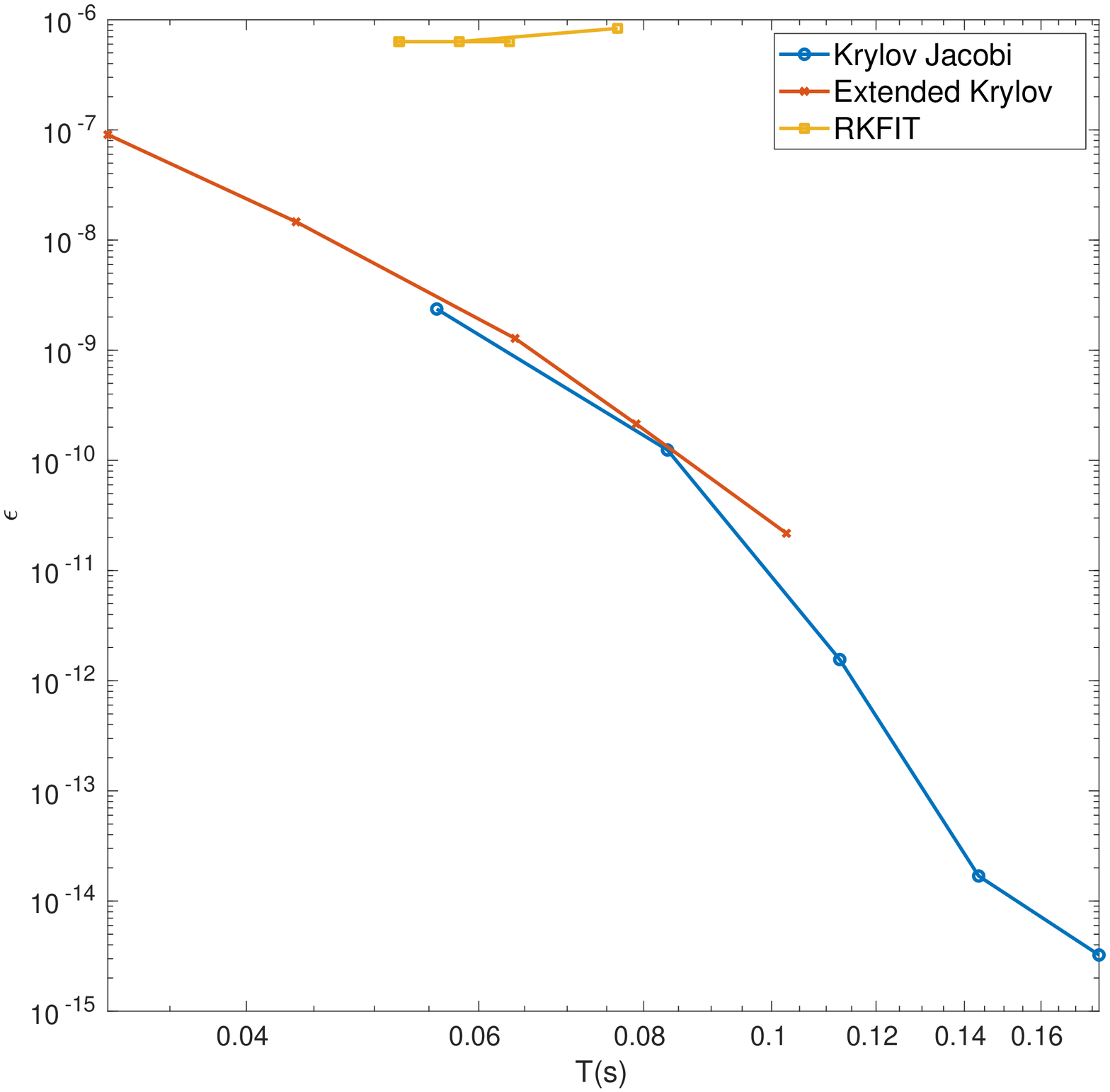}
    }
    \subfloat[$\alpha=1.5$]{
       \includegraphics[width=0.32\columnwidth]{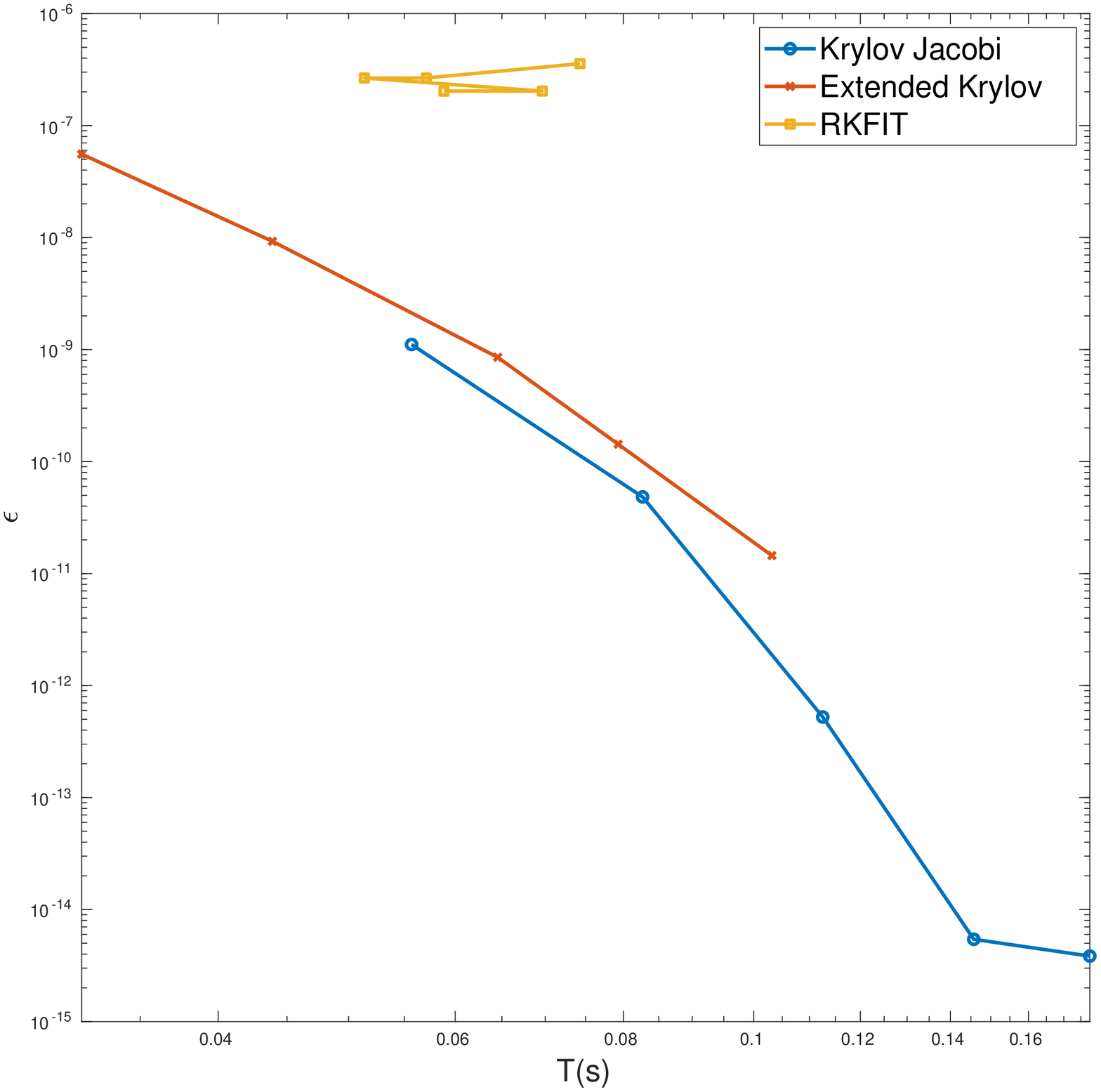}}
    \subfloat[$\alpha=1.8$]{
       \includegraphics[width=0.32\columnwidth]{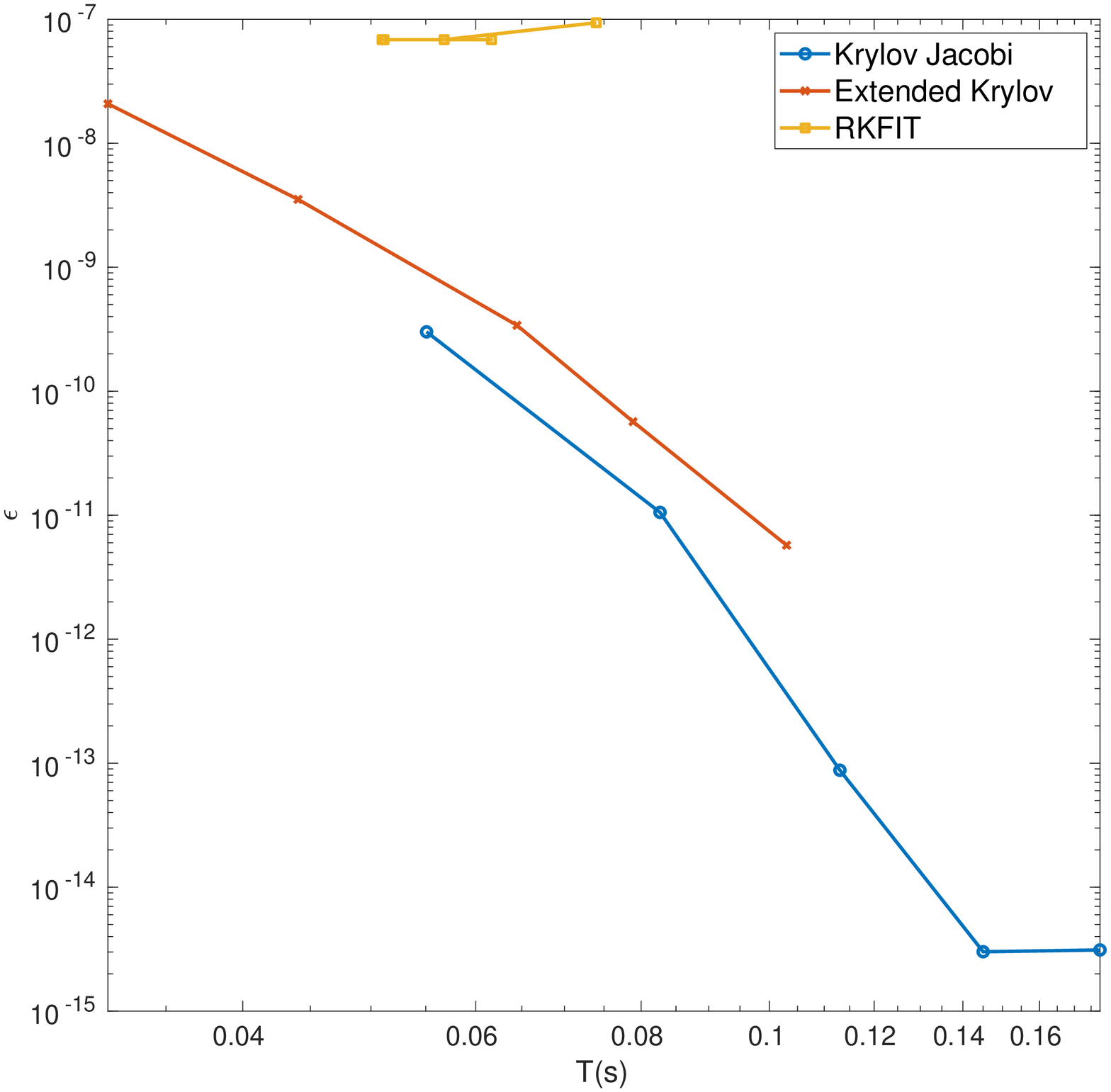}}
    \caption{P1-- FV: The nodes in the time--error graph corresponds to $k =  10,15,20,25,30$ poles.}\label{fig:f0_FV}
\end{figure}

\paragraph{P1 by using FEM}

Consider here the FEM discretization of
problem~\eqref{eq:time_dependent} on the structured mesh made by
triangular cells with straight sides, i.e., each side of the squared
domain $\Omega$ is divided into $n_x$ and $n_y$ rectangles, and then
divided into a pair of triangles, with Lagrangian elements of order
$1$. From Figure~\ref{fig:f0_FEM} we observe that the RKFIT
algorithm produces again less than the requested number of nodes and
at a greater computational effort. Considering Krylov spaces of the
same size, the Krylov Jacobi  method   is more
accurate than the extended Krylov. Moreover, the latter in this
example is only marginally better with respect to the execution
time.
\begin{figure}[htbp]
    \centering
    \subfloat[$\alpha=1.2$]{
       \includegraphics[width=0.32\columnwidth]{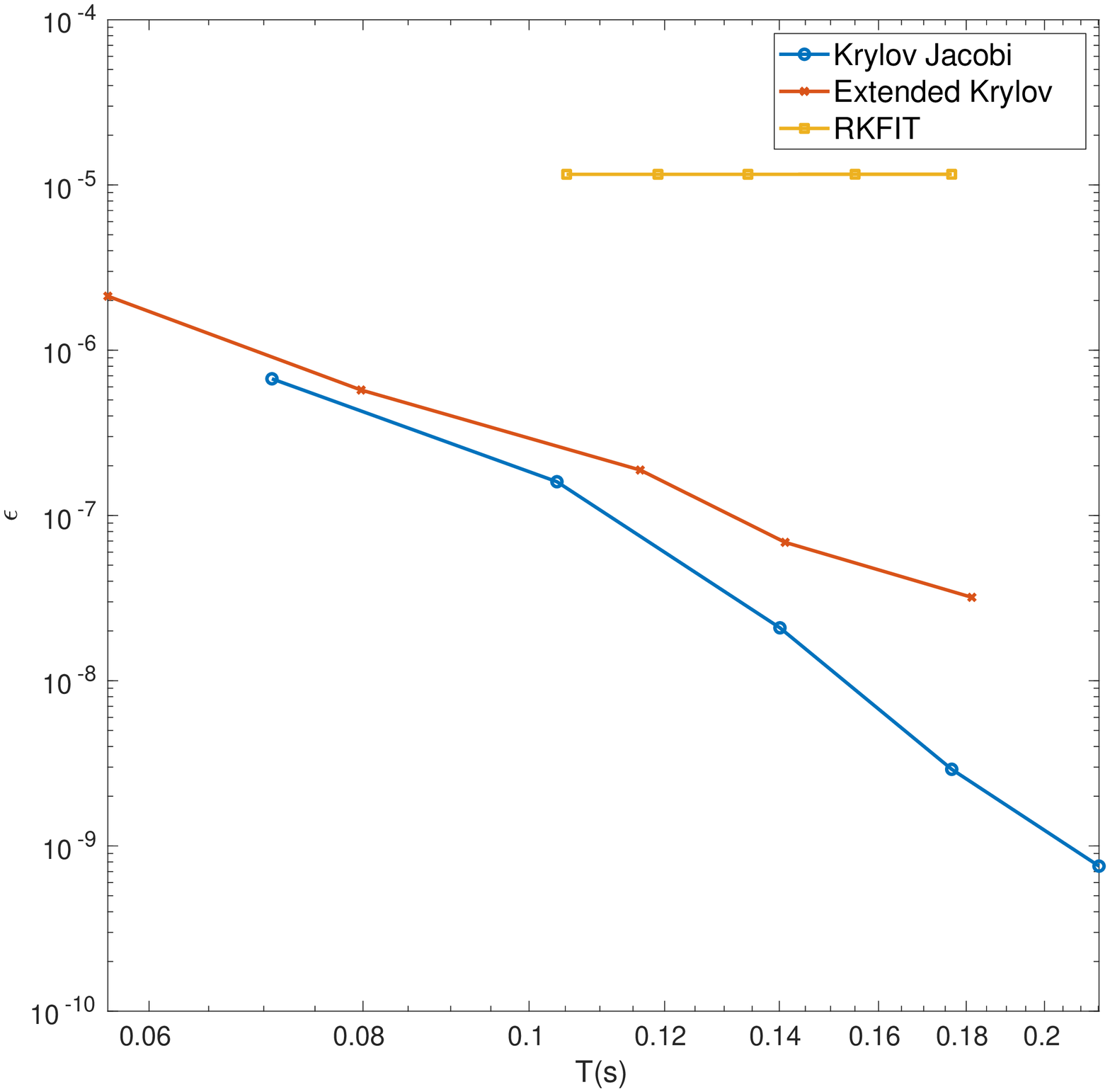}
    }
    \subfloat[$\alpha=1.5$]{
     \includegraphics[width=0.32\columnwidth]{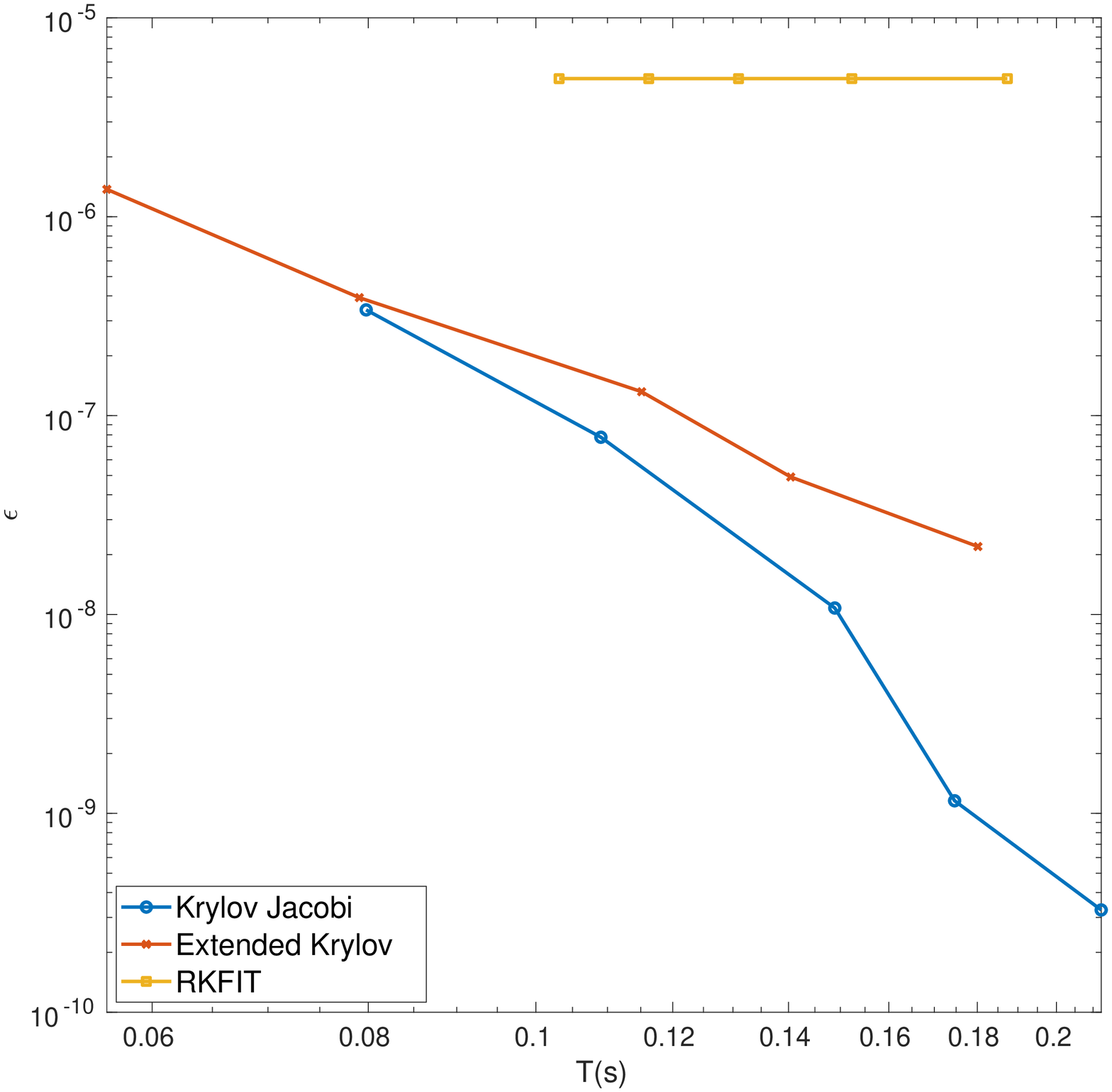}}
    \subfloat[$\alpha=1.8$]{
        \includegraphics[width=0.32\columnwidth]{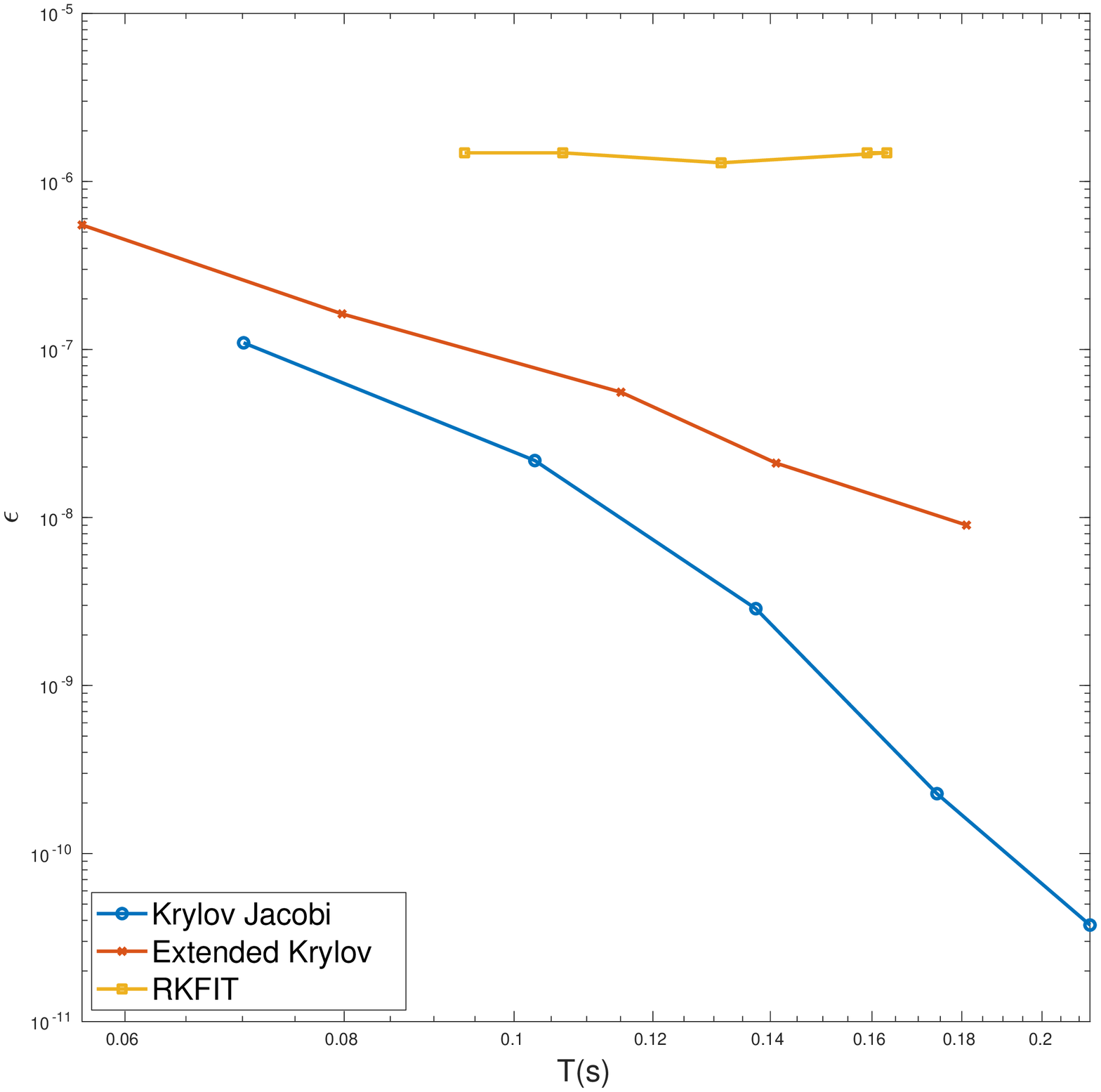}
    }
    \caption{P1-- FEM: The nodes in the time--error graph corresponds to $k =  10,15,20,25,30$ poles.}\label{fig:f0_FEM}
\end{figure}

\paragraph{P2 by using FV}
Consider the cell--centered FV for~\eqref{eq:allen-cahn}
 over a uniform quadrangular mesh with $n_x = n_y =
64$ intervals on the unit square $[0,1]^2.$
To solve this problem we consider the application of the IMEX
backward Euler method with $\delta_t = 10^{-2}.$
\begin{figure}[htbp]
    \centering
    \subfloat[$\alpha = 1.2$]{\includegraphics[width=0.33\columnwidth]{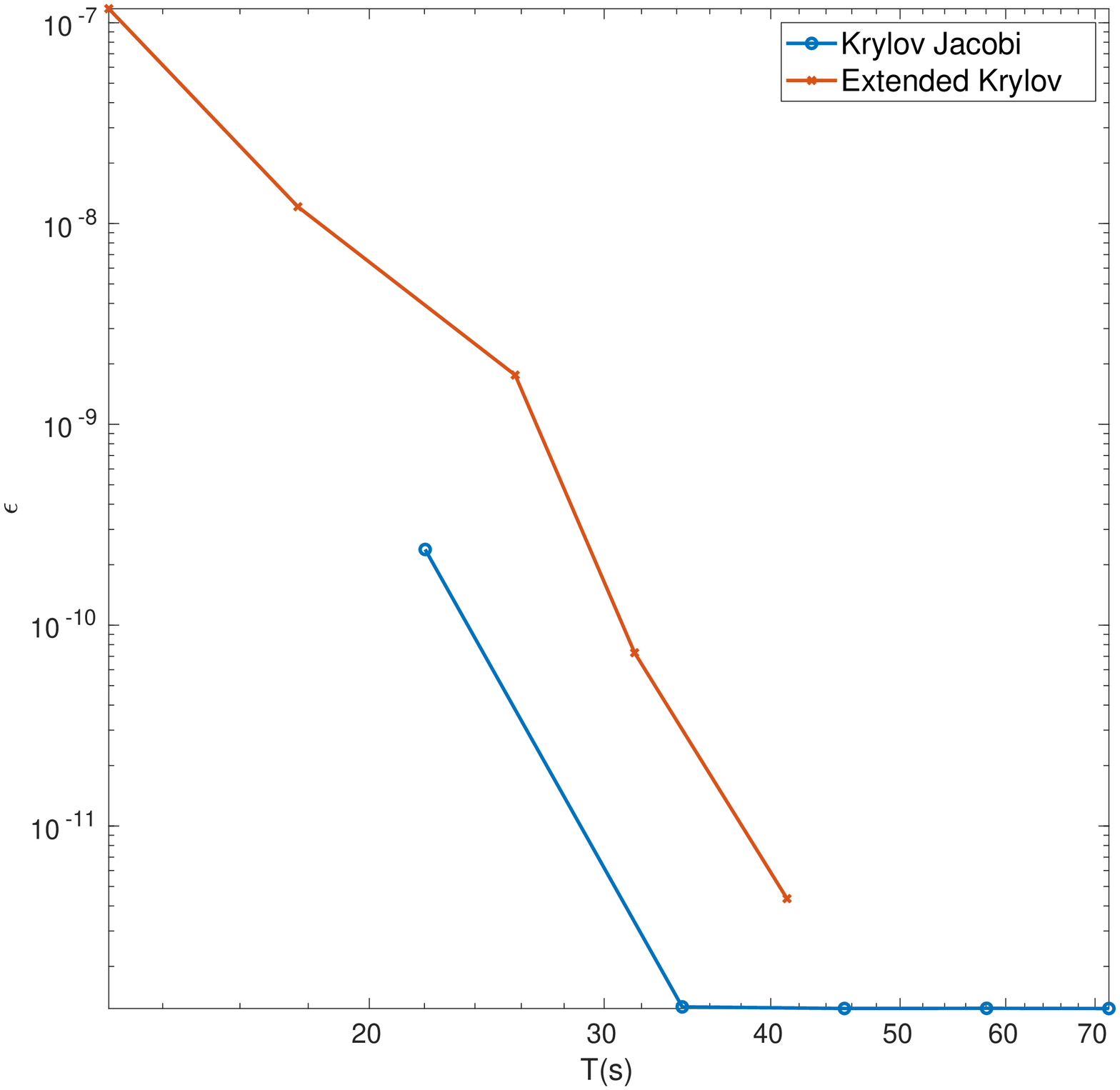}}
    \subfloat[$\alpha = 1.5$]{\includegraphics[width=0.33\columnwidth]{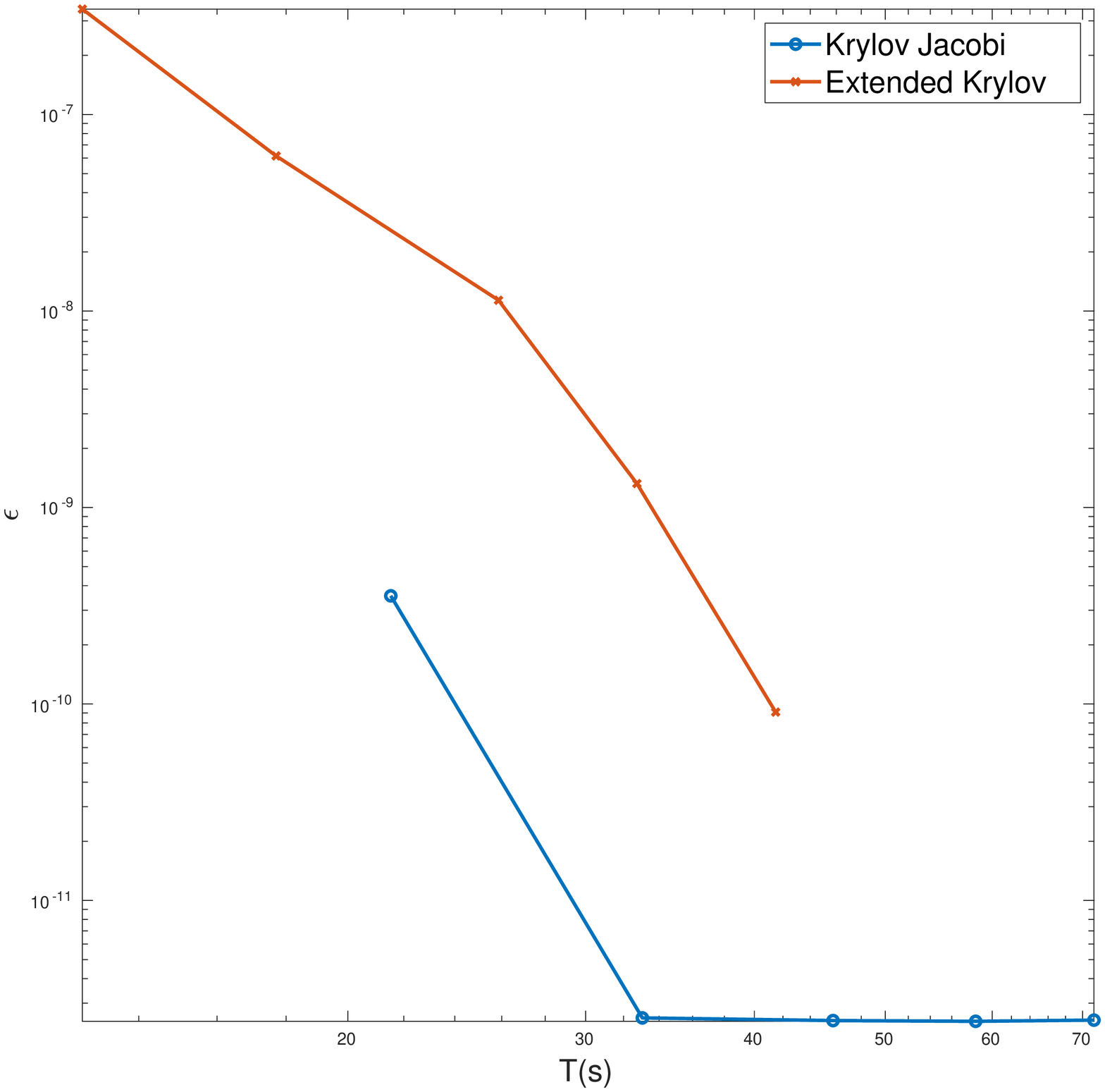}}
    \subfloat[$\alpha = 1.8$]{\includegraphics[width=0.33\columnwidth]{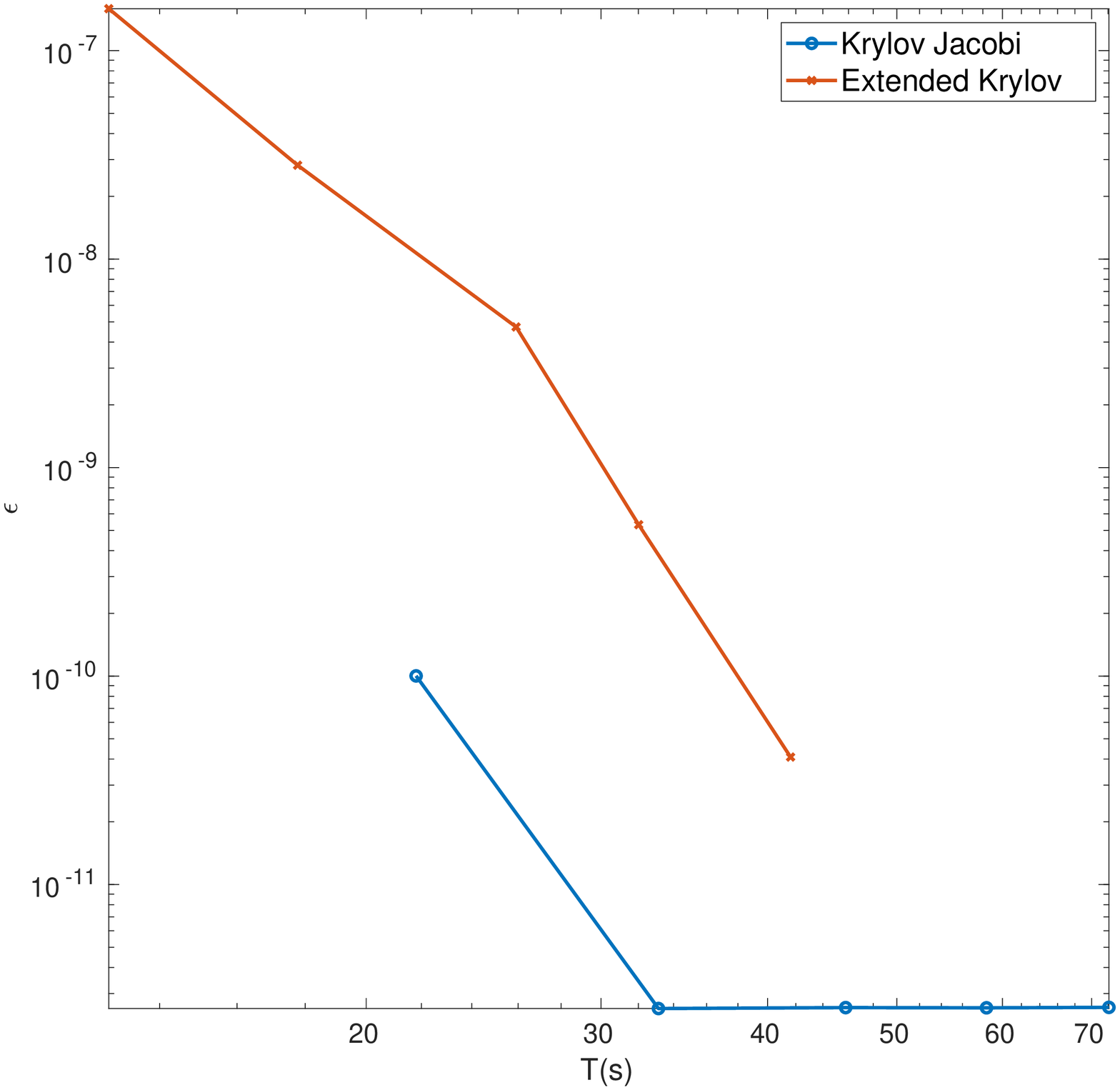}}
    \caption{P2-- FV:  Allen--Cahn equation~\eqref{eq:allen-cahn} with $\mu = 10^{-4}.$ The nodes in the time--error graph corresponds to $k =  10,15,20,25,30$ poles.}\label{fig:allen_cahn_fv}
\end{figure}
The results are
collected in Figure~\ref{fig:allen_cahn_fv}, from which we observe
that the  Krylov Jacobi method  achieves the
better accuracy already with $k = 5$ nodes. The RKFIT Algorithm,
tuned on the problem with $n_x = n_y = 32$, fails to compute
acceptable nodes in this case, see the example of the
errors for $\alpha = 1.8$ given in Figure~\ref{fig:allen_cahn_solution_fv}.
\begin{figure}[htbp]
    \centering
    \includegraphics[width=\columnwidth]{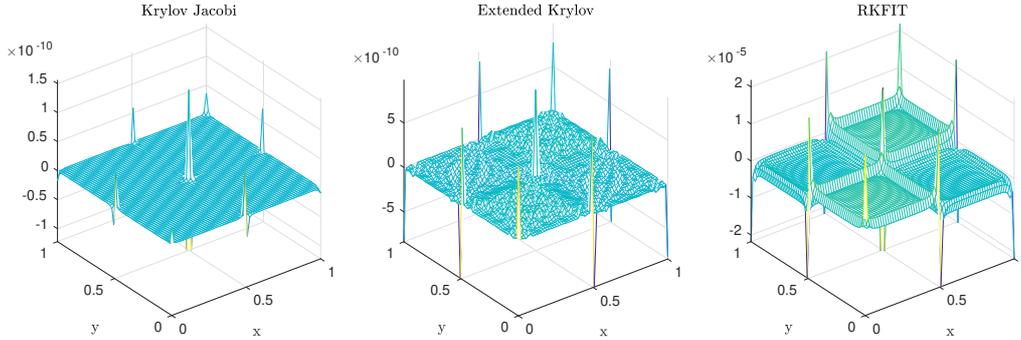}
    \caption{P2-- FV:  Comparison of the relative error for the Allen--Cahn equation~\eqref{eq:allen-cahn} with $\mu = 10^{-2}.$ }
\label{fig:allen_cahn_solution_fv}
\end{figure}

\paragraph{P2 by using FEM}
Consider the discretization of~\eqref{eq:allen-cahn} by means of
Lagrange linear elements over a uniform finite element mesh on the
unit square $[0,1]^2,$  i.e., a mesh consisting of triangular cells
with straight sides dividing each side of the square into $n_x$ and
$n_y$ rectangles, each divided into a pair of triangles, for $n_x =
n_y = 80.$ The integration method in time is the IMEX backward Euler
with $\delta_t = 10^{-2}.$ The performance of the considered
rational Krylov method are shown in Figure~\ref{fig:allen_cahn_fem}.
Note that the behavior is analogous to the one observed for the
time--dependent problem~\eqref{eq:time_dependent} (compare the
results with those in Figure~\ref{fig:f0_FEM}).
\begin{figure}[htbp]
    \centering
    \subfloat[$\alpha = 1.2$]{\includegraphics[width=0.33\columnwidth]{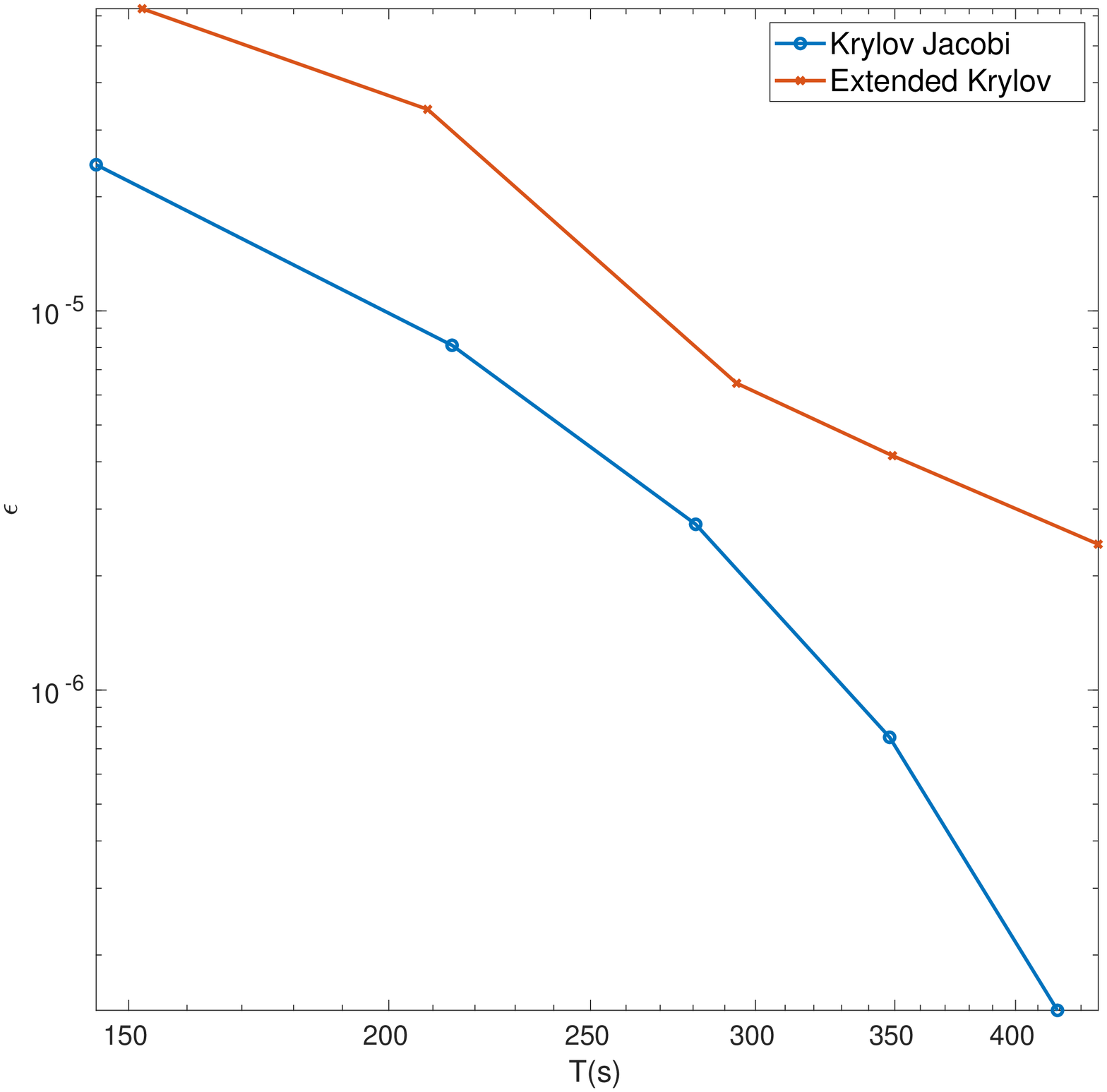}}
    \subfloat[$\alpha = 1.5$]{\includegraphics[width=0.33\columnwidth]{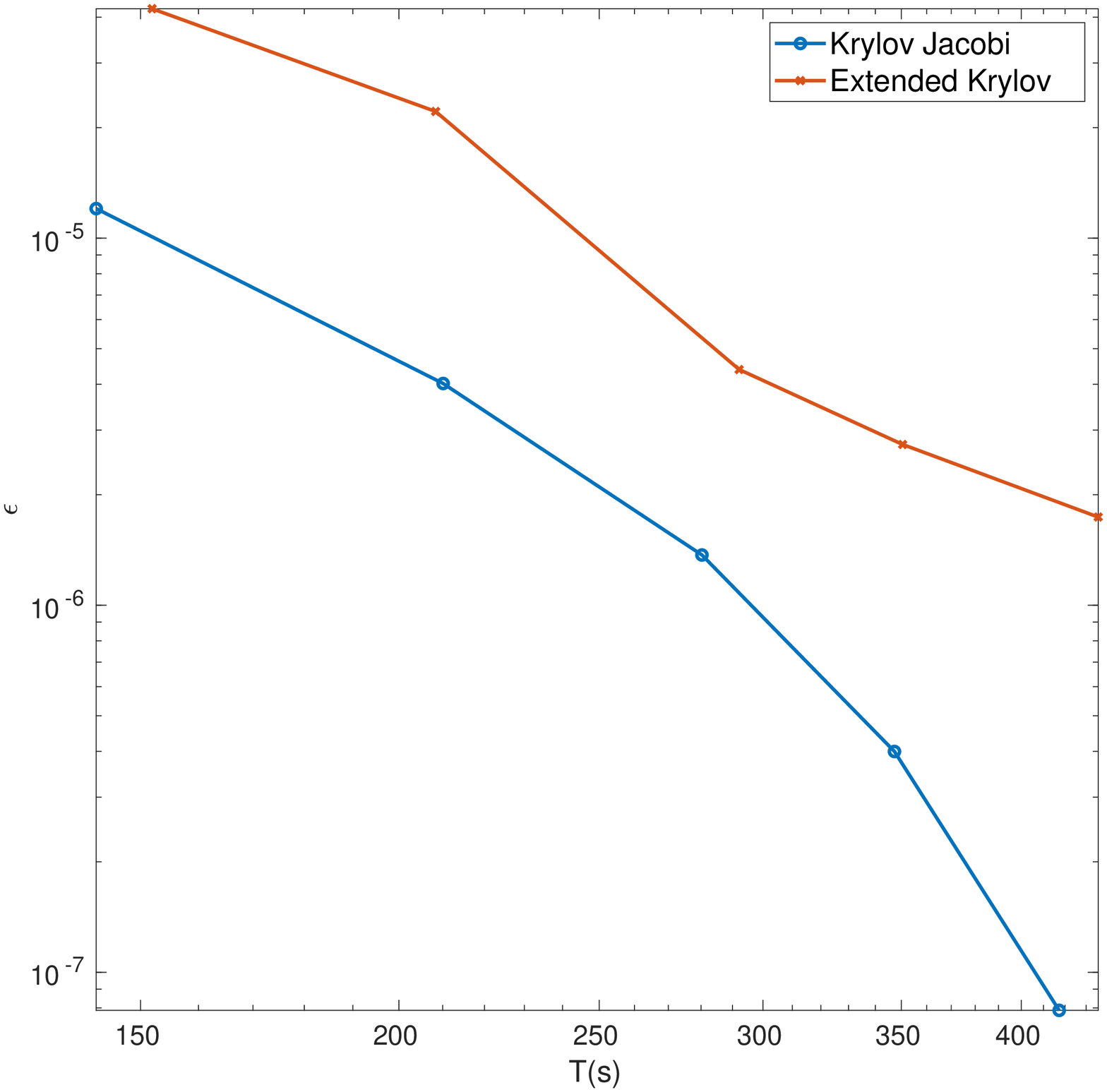}}
    \subfloat[$\alpha = 1.8$]{\includegraphics[width=0.33\columnwidth]{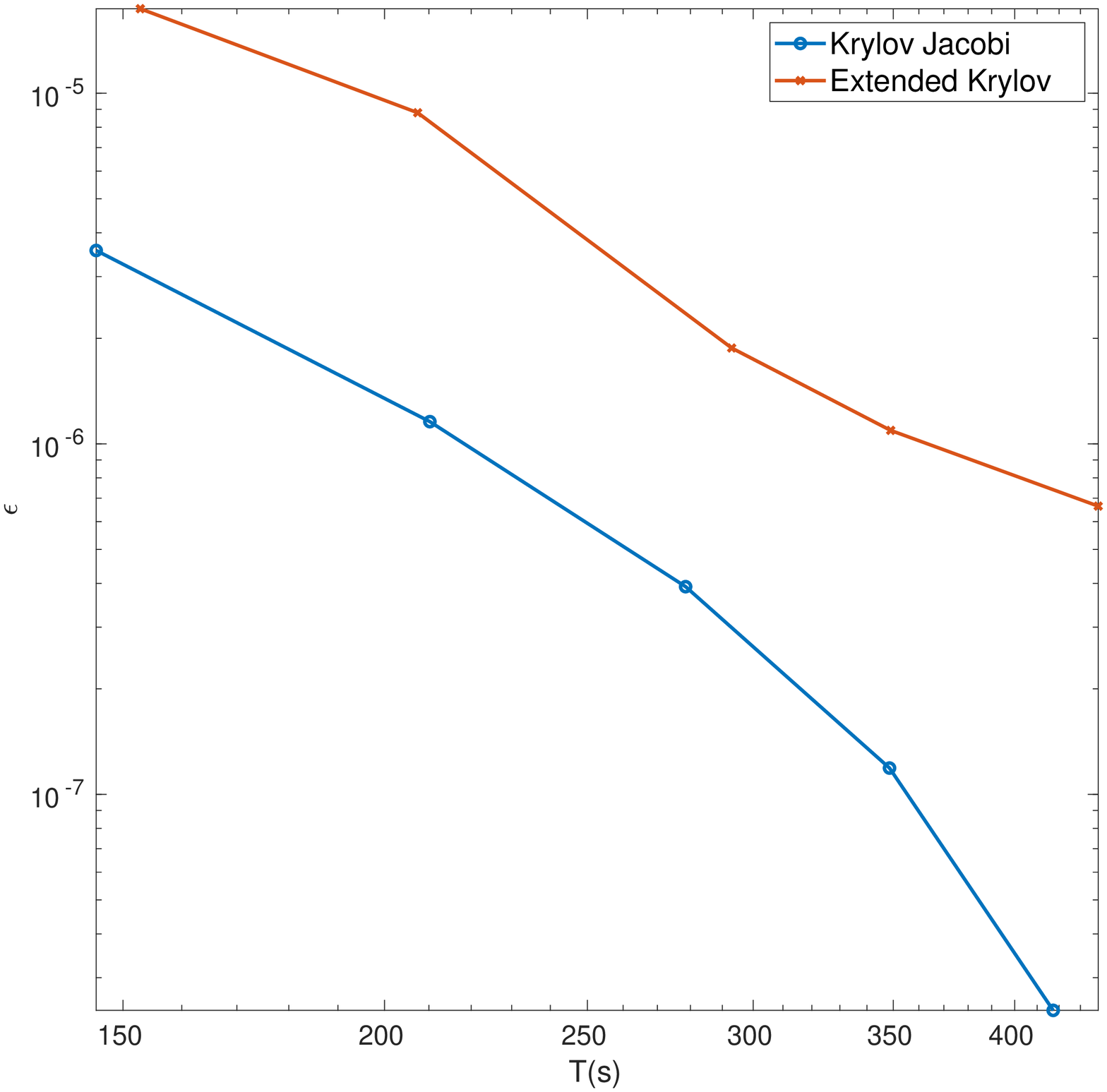}}
    \caption{P2-- FEM:  Allen--Cahn equation~\eqref{eq:allen-cahn} with $\mu = 10^{-3}.$ The nodes in the time--error graph corresponds to $k =  10,15,20,25,30$ poles.}
    \label{fig:allen_cahn_fem}
\end{figure}

\section{Conclusion and perspectives}

We considered certain matrix functions involving
the fractional power of real positive definite matrices computed by
rational Krylov methods with application to the numerical
approximation of fractional-in-space partial differential equations.
We proposed a new choice of the poles, real and simple by
construction, that appears to be effective and faster with respect
to polynomial, shift-and-invert, extended Krylov methods, and RKFIT.

Within the same framework, one can apply the proposed rational
Krylov Jacobi method also to \emph{fractional linear multistep
methods}~\cite{MR838249} for solving~\eqref{eq:time_dependent} 
in which the first order derivative in
time is replaced by a fractional derivative of order $\beta,$ with
$\beta \in (0,1).$

In principle, the proposed method can work on irregular domains, 
	Robin or Dirichlet BC, see, e.g., \cite{Gibou1, Gibou2, Gibou3}, and can 
	be used also in contexts of adaptivity, provided that these generate a symmetric 
	positive definite matrix. In the latter case, studying how the selected poles 
	vary as the mesh is changed would be of interest, and could also open 
	new alternative approaches.

%\section*{Bibliography}

\bibliographystyle{elsarticle-num}
\bibliography{fractionalrationalkrylov.bib}

\end{document}